\documentclass[onefignum,onetabnum]{siamart}
\usepackage{amsmath,amssymb}
\usepackage{caption}
\usepackage[utf8]{inputenc}
\usepackage{mathtools}
\usepackage{mathrsfs}
\usepackage{color,graphicx}
\usepackage{thmtools, thm-restate}
\usepackage{enumitem}

\usepackage{etoolbox}
\makeatletter

\newsiamremark{expl}{Example}
\newcommand{\cG}{\mathcal{G}}

\newcommand{\cO}{\mathcal{O}}

\newcommand{\cM}{\mathcal{M}}

\newcommand{\cP}{\mathcal{P}}

\newcommand{\brho}{\boldsymbol{\rho}}

\newcommand{\bitm}{\begin{itemize}}
\newcommand{\eitm}{\end{itemize}}
\newcommand{\bitme}{\begin{enumerate}[label=(\roman*),leftmargin=0.4in]}
\newcommand{\eitme}{\end{enumerate}}
\newcommand{\beq}{\begin{equation}}
\newcommand{\eeq}{\end{equation}}
\def\bals#1\eals{\begin{align*} #1 \end{align*}}
\def\bal#1\eal{\begin{align} #1 \end{align}}

\newcommand{\Iif}{\quad \text{ if } }

\newcommand{\Aand}{\quad \text{ and } \quad }

\newcommand\Dom\Omega

\newcommand\CC{\mathbb{C}}

\newcommand\PP{\mathbb{P}}
\newcommand\RR{\mathbb{R}}

\newcommand\TT{\mathbb{T}}

\newcommand\NN{\mathbb{N}}

\newcommand\UU{\mathbb{U}}
\newcommand\VV{\mathbb{V}}
\newcommand\WW{\mathbb{W}}

\newcommand\Lap\Delta

\newcommand\abs[1]{\left\lvert #1 \right\rvert}

\newcommand\dx{\,\mathrm{d}x}

\newcommand\dtau{\,\mathrm{d}\tau}

\def\bpde#1\epde{\[\left\{\begin{aligned}#1\end{aligned}\right. \]}
\def\inbpde#1\inepde{\left\{\begin{aligned}#1\end{aligned}\right.}
\def\binpde#1\einpde{\left\{\begin{aligned}#1\end{aligned}\right.}

\newcommand\Normlr[2]{\left\lVert { #1 } \right\rVert_{#2}}
\newcommand\Norm[2]{\lVert { #1 } \rVert_{#2}}

\def\bw{\mathbf{w}}

\def\cB{\mathcal{B}}

\def\cD{\mathcal{D}}

\def\cW{\mathcal{W}}

\def\half{\frac{1}{2}}

\def\bw{\mathbf{w}}

\def\p{\partial}

\def\bfB{\mathbf{B}}
\def\bfC{\mathbf{C}}

\def\bfN{\mathbf{N}}
\def\bfM{\mathbf{M}}
\def\bfP{\mathbf{P}}

\def\bfP{\mathbf{P}}
\def\bfU{\mathbf{U}}
\def\bfV{\mathbf{V}}
\def\bfW{\mathbf{W}}

\def\bfU{\mathbf{U}}
\def\bfV{\mathbf{V}}

\def\bfGamma{\mathbf{\Gamma}}
\def\bfTheta{\mathbf{\Theta}}

\def\b0{\mathbf{0}}

\def\bb{\mathbf{b}}
\def\bc{\mathbf{c}}

\def\bfXi{\boldsymbol{\Xi}}
\def\bmu{\boldsymbol{\mu}}

\def\eps{\varepsilon}

\def\bbmat{\begin{bmatrix}[r]}
\def\ebmat{\end{bmatrix}}

\newcommand{\barr}{\begin{array}}
\newcommand{\ea}{\end{array}}
\newcommand{\bea}{\begin{eqnarray}}
\newcommand{\eea}{\end{eqnarray}}
\newcommand{\bt}{\begin{table}}
\newcommand{\et}{\end{table}}

\DeclareMathOperator\Id{Id}

\DeclareMathOperator\supp{supp}

\theoremstyle{plain}

\newsiamthm{cond}{Condition}
\newsiamremark{remark}{Remark} 
\newsiamremark{example}{Example} 

\newsiamthm{assumption}{Assumption}

\numberwithin{equation}{section}

\newcommand\tturl[1]{{\tt \scriptsize [\url{{#1}}]}}

\newcommand{\bfb}{\boldsymbol b}

\newcommand{\bfxi}{\boldsymbol \xi}

\newcommand{\bfz}{\boldsymbol z}

\newcommand{\bfgamma}{\boldsymbol \gamma}

\newcommand{\Mcal}{\mathcal{M}}

\def\bfXi{\boldsymbol \Xi}
\newcommand{\bfy}{\boldsymbol y}

\def\tfin{t_{\textrm{F}}}

\newcommand\cMf{\mathcal{M}^{(\mathrm{f})}}
\newcommand\cMr{\mathcal{M}^{(\mathrm{r})}}
\newcommand\cMfl{\cMf_S}

\newcommand\uf{u_{N_\delta}^{(\textrm{f})}}
\newcommand\ur{u^{(\textrm{r})}_M}

\newcommand\cMfs{\underline{\mathcal{M}}^{(\mathrm{f})}}
\newcommand\cMrs{\underline{\mathcal{M}}^{(\mathrm{r})}}

\newcommand\ffsn{\underline{f}_{N_\delta}}

\newcommand\ufsn{\underline{u}_{N_\delta}}

\newcommand\frsM{\underline{f}^{(r)}_{M}}

\newcommand\urs{\underline{u}^{(r)}}
\newcommand\uos{\underline{u}_{0}}
\newcommand\uosl[1]{\underline{u}_{0,#1}^{(r)}}

\newcommand\ursM{\underline{u}^{(r)}_{M}}

\newcommand\cNs{\underline{\mathcal{N}}}
\newcommand\cNrs{\underline{\mathcal{N}}^{(r)}}

\newcommand\cNd{\overline{\mathcal{N}}}
\newcommand\cNrd{\overline{\mathcal{N}}^{(r)}}
\newcommand\cNdS{\overline{\mathcal{N}}_S}

\newcommand\cMfd{\overline{\mathcal{M}}^{(\mathrm{f})}}
\newcommand\cMrd{\overline{\mathcal{M}}^{(\mathrm{r})}}

\newcommand\ffdP{\overline{f}_{\bfP}}
\newcommand\ffdPi{\overline{f}_{\bfP}^{(i)}}

\newcommand\ufdP{\overline{u}_{\bfP}}

\newcommand\frdXi{\overline{f}^{(r)}_{\bfXi}}

\newcommand\urd{\overline{u}^{(r)}}

\newcommand\urdXi{\overline{u}^{(r)}_{\bfXi}}

\newcommand\Linv{L_{\textrm{inv}}}

\newcommand\NormLzdl[1]{\Normlr{#1}{L^2(\Dom_\ell)}}

\newcommand{\Tlzs}{\underline{T}_{12}}
\newcommand{\Tls}{\underline{T}_{1}}

\newcommand\cNrmats{\cNrd_{\textrm{MATS}}}
\newcommand\cNmats{\cNd_{\textrm{MATS}}}

\definecolor{darkgreen}{rgb}{0.01, 0.75, 0.24}%

\newcommand{\TheTitle}{\textsf{{\small Depth separation for reduced deep networks in nonlinear model reduction: Distilling shock waves in nonlinear hyperbolic
problems}}}

\begin{document}

\ifpdf
\DeclareGraphicsExtensions{.pdf, .jpg, .tif}
\else
\DeclareGraphicsExtensions{.eps, .jpg}
\fi

\title{\TheTitle}

\author{Donsub Rim%
  \thanks{Courant Institute, %
  New York University, New York, NY 10012 %
  (\email{{\tt dr1653@nyu.edu}},
   \email{{\tt venturi@cims.nyu.edu}},
   \email{{\tt bruna@cims.nyu.edu}},
   \email{{\tt pehersto@cims.nyu.edu}},
  )}%
 \and
    Luca Venturi\footnotemark[1]%
 \and
    Joan Bruna\footnotemark[1]%
 \and
    Benjamin Peherstorfer\footnotemark[1]%
}
\maketitle

\begin{abstract}
Classical reduced models are low-rank approximations using a fixed basis designed to achieve dimensionality reduction of large-scale systems.
In this work, we introduce reduced deep networks, a generalization of classical reduced models formulated as deep neural networks. We prove depth separation results showing that reduced deep networks  approximate solutions of parametrized hyperbolic partial differential equations with approximation error $\eps \in (0,1)$ with $\cO(|\log(\eps)|)$ degrees of freedom, even in the nonlinear setting where solutions exhibit shock waves. We also show that classical reduced models achieve exponentially worse approximation rates by establishing lower bounds on the relevant Kolmogorov $N$-widths.
\end{abstract}



\section{Introduction}

We propose \emph{reduced deep networks} (RDNs), which are deep neural network (DNN) constructions that generalize  classical reduced models \cite{crb-book,siamrev-survey}. We show that RDNs achieve exponentially faster error decay with respect to number of degrees of freedom when approximating solution manifolds of certain nonlinear hyperbolic partial differential equations (PDEs) in contrast to classical reduced models. 
Our arguments yield lower bounds on the smallest number of degrees of freedom necessary to achieve a given accuracy with classical reduced models, by estimating the Kolmogorov $N$-width \cite{pinkus12,crb-book}. The lower bounds apply in general to a function class we call \emph{sharply convective} and advances the existing results \cite{Ohlberger16,Greif19,welper17} beyond constant-speed problems. The two results indicate a type of depth separation: RDNs can achieve dimensionality reduction where shallow approximations such as classical reduced models cannot. The results are shown for representative hyperbolic problems, the color equation (variable-speed transport) and the Burgers' equation in a single spatial dimension.

Classical reduced models fail to be efficient not only for hyperbolic problems but for transport-dominated problems in general \cite{rowley00,Ohlberger13}. Nonlinear model reduction techniques are developed to overcome the limitations. These include the removal of symmetry \cite{rowley00}, dynamical low-rank (DLR) approximations or dynamically orthogonal (DO) method \cite{koch07,sapsis09,musharbash20}, method of freezing \cite{Ohlberger13}, approximated Lax-Pairs \cite{Gerbeau14}, reduction of optimal transport maps \cite{iollo14}, calibrated manifolds \cite{Cagniart2019,nonino19}, shock curve estimation \cite{taddei14}, adaptive online low-rank updates \cite{pehersto15,P18AADEIM}, adaptive $h$-refinement \cite{carlberg15}, shifted proper orthogonal decomposition (sPOD) \cite{schulze18}, Lagrangian basis method \cite{Mojgani17}, transport reversal \cite{rim17reversal}, transformed snapshot interpolation \cite{welper17,welper19}, generalized Lax-Philips representation \cite{rim18mr,radonsplit} deep autoencoders \cite{Lee20}, characteristic dynamic mode decomposition \cite{sesterhenn2019}, registration methods \cite{taddei19}, Wasserstein barycenters \cite{ehrlacher19}, unsupervised traveling wave identification with shifting truncation \cite{mendible19},  a generalization of the moving finite element method (MFEM) \cite{black19}, and Manifold Approximations via Transported Subspaces (MATS) \cite{Rim19}.  A common feature among these new methods is the dynamic adaptation of the low-rank representation. The adaptation is achieved using low-rank updates, adaptive refinements, or nonlinear transformations. 

The works \cite{welper19,Lee20} make use of DNNs. There also has been efforts to approximate the solution manifold of parametric PDEs directly with DNNs \cite{Kutyniok19,Regazzoni19,Laakmann20,Geist20}, by exploiting the expressive power of DNNs for approximating solutions of PDEs and nonlinear functions in general \cite{Cybenko89,Telgarsky16,Yarotsky17,Raissi19,Daubechies19,Schwab19}. DNNs also have been used to compute the reduced coefficients \cite{Wang19}. The key challenge in these approaches is in achieving the level of computational efficiency desired in model reduction, as these DNN constructions are more computationally expensive to evaluate or manipulate than the classical reduced models.

MATS is a nonlinear reduced solution that is written as a composition of two low-rank representations, which allows efficient computations. The efficiency is equivalent to that of classical reduced models and thus enables it to be used directly with the governing differential equations and achieve significant speed-ups \cite{Rim19}. MATS was motivated by the distinguishing feature of hyperbolic PDEs, namely that the solution propagates along characteristic curves \cite{evans10,fvmbook}. However, there are limitations in its applicability, as the numerical experiments in \cite{Rim19} indicate that the efficiency of MATS depends on the regularity of the characteristic curves. 

The RDN introduced here is a generalization of MATS with additional hidden layers, where each layer has a low-rank representation. We will show that RDNs yield efficient approximations of singular characteristic curves by using additional hidden layers with regular representations. Thus, RDNs can approximate solution manifolds of nonlinear hyperbolic PDEs, even when nonlinear shocks are present.

The RDN is reminiscent of the compression framework for deep networks that is being studied theoretically for improving generalization bounds \cite{Neyshabur17,Arora18}, or being utilized in practice to accelerate the performance of large networks in practical applications \cite{chen15,novikov15,Cheng18}. However, the fact that an RDN is a set of networks with a specifically designed degree of freedom, rather than a single network exhibiting low-rank structure in its weights, distinguishes it from the compression frameworks. Furthermore, the specific architecture we use includes special components, such as layers that compute the inverse of a function, not very common in generic architectures used in machine learning.

RDNs are different from deep network approximations that have sparse connections \cite{Bolcskei19,LeCun89}. An RDN can be viewed as a dense network with a large number of activations, albeit with very few number of effective parameters. But beyond the differences in the architecture, the RDNs are constructed to maintain important properties that are indispensible in model reduction. While sparse approximations lead to efficient approximations of general function classes \cite{temlyakov08}, such approximations are difficult to deploy in model reduction applications. For example, the choice of $N$ best terms is not necessarily regular with respect to the target of approximation, whereas the success of the reduced system rely crucially on such regularity.

In the machine learning literature, \emph{distillation} or \emph{model compression} refers to the transferring of the learned knowledge from an accurate model to another specialized model that is more efficient for deployment \cite{Bucilua06,Hinton15D}. Model reduction is driven by an identical motivation.

\clearpage

\section{Reduced deep networks}

In this section, we introduce RDNs and the notion of deep reduction. We first provide a brief overview of model reduction for computing reduced solutions and then show that reduced solutions can be represented as shallow networks. We then derive a deep-network representations of reduced solutions, resulting in RDNs.

\subsection{Model reduction} \label{sec:model-reduction}

We give a brief overview of model reduction. For a comprehensive review, we refer the reader to the references \cite{siamrev-survey,crb-book}.

Our goal is in approximating solutions of PDEs. The specific PDEs will be defined later. For now, it is only important that the solution functions $u$ depend on the spatial variable $x$, time $t$, and parameters $\bmu$. Let us denote by $\cM$ the \emph{solution manifold}, 
\beq
    \cM := \{u(\cdot,t;\bmu) \in \VV: \Dom \to \RR 
    \,|\, t \in [0,\tfin], \bmu \in \cD \}\,,
    \label{eq:manifold}
\eeq
which is a set of functions in a real Hilbert space $\VV := L^2(\Dom)$ over the spatial domain $\Dom:=(0,1)$. The parameter domain is $\cD \subset \RR^P$ ($P \in \NN$) and the time interval is $[0,\tfin]$, where $\NN$ denotes the set of natural numbers. 

A \emph{full solution} (or a \emph{full-model solution}) is an approximation of a solution $u \in \cM$ in a finite-dimensional subspace spanned by $N_{\delta} \in \mathbb{N}$ basis functions $\{\varphi_n\}_{n=1}^{N_\delta} \subset \VV$,
\beq
    \uf(x; t, \bmu) = \sum_{n=1}^{N_\delta} w_n(t, \bmu) \varphi_n(x),
    \label{eq:full}
\eeq
with coefficients $\{w_n(t, \bmu)\}_{n = 1}^{N_{\delta}}$ that depend on time and parameter. For ease of exposition, we consider in the following full solutions that are piecewise linear in the spatial variable $x$ on an equidistant grid with $N_{\delta}$ grid points and $\{\varphi_n\}_{n = 1}^{N_{\delta}}$ being the canonical nodal point basis \cite{Strang73}. Then, for all $\delta \in (0,1)$, there is $N_\delta$ large enough so that for each $(t,\bmu) \in [0,\tfin] \times \cD$ the full solution $\uf(\cdot;t,\bmu)$ of the form \cref{eq:full} approximates the solution $u(\cdot;t,\bmu) \in \cM$ with
\beq
    \Normlr{u(\cdot;t,\bmu) - \uf(\cdot;t,\bmu)}{\VV} < \delta\,.
    \label{eq:full-accuracy}
\eeq
For a fixed $N_{\delta}$, the \emph{approximate solution manifold} is
\beq
\cMf
:=
\left\{
    \uf(\cdot;t,\bmu): (t,\bmu) \in [0,\tfin] \times \cD 
\right\}.
\label{eq:full-manifold}
\eeq
Full solutions typically are computed with finite-difference, finite-element or finite-volume methods, which can be computationally expensive if a large $N_{\delta}$ is required to achieve the desired tolerance $\delta$. 
Model reduction aims to construct reduced solutions in problem-dependent subspaces of much lower dimension $M \ll N_{\delta}$ to reduce computational costs \cite{siamrev-survey,crb-book}. 
Model reduction consists of an \emph{offline stage} and an \emph{online stage}. During the offline stage, the basis of the low-dimensional subspace, the reduced space $\VV_{M}$, is constructed. 
A reduced basis is typically computed by collecting a finite subset $\cMfl = \{\uf(\cdot;t_i,\bmu_i) \}_{i=1}^S \subset \cMf$ of full solutions, where  $S \in \NN$ and $\{(t_i,\bmu_i)\}_{i=1}^S \subset [0,\tfin]\times \cD$, and then computing a low-dimensional basis using, e.g., the singular value decomposition (SVD) \cite{golub96}.
Let $\{\xi_m\}_{m=1}^{M} \subset \VV$ be the set of the reduced-basis functions.

In the online phase, a \emph{reduced solution} (or a \emph{reduced-model solution}) is derived in the space spanned by the reduced basis, 
\beq
    \ur (x;t,\bmu)
    :=
    \sum_{m=1}^{M} \gamma_m(t,\bmu) \xi_m(x).
    \label{eq:reduced}
\eeq
The coefficients $\{\gamma_m(t,\bmu)\}_{m=1}^M$ of the reduced solutions are obtained by solving a system of equations for any given $(t, \bmu) \in [0, \tfin] \times \cD$. 
The reduced system is derived using the PDE. The computational complexity of solving the reduced system scales with the dimension of the reduced space $M$ and is independent of the dimension of the full solutions $N_\delta$. If the dimension $M$ of the reduced space is small compared to the dimension $N_\delta$ of the full solutions, then solving for the reduced solution can be computationally cheaper than solving for the full solution. At the same time, the dimension $M$ of the reduced space needs to be chosen sufficiently large so that the reduced solution are sufficiently accurate.  

Analogously to \cref{eq:full}, we assume in the following that for all $\eps \in (0,1)$, there exists $M \in \NN$ such that for each $(t,\bmu) \in [0,\tfin] \times \cD$, the solution $u(\cdot;t,\bmu) \in \cM$ can be approximated with a reduced solution $\ur(\cdot;t,\bmu)$ of the form \cref{eq:reduced} satisfying
\beq
    \Normlr{u(\cdot;t,\bmu) - \ur(\cdot;t,\bmu)}{\VV} < \eps.
    \label{eq:reduced-accuracy}
\eeq
Note that in the model reduction literature, the error \cref{eq:reduced-accuracy} is typically obtained with respect to the full solution $\uf$, rather than the (exact) solution $u$. 
For a fixed reduced basis $\{\xi_m\}_{m = 1}^M$ with $M$ basis functions, we call the set of reduced solutions $\cMr$ that satisfies \cref{eq:reduced-accuracy} the \emph{reduced solution manifold},
\beq
    \cMr
    :=
    \left\{
        \ur(\cdot;t,\bmu): (t,\bmu) \in [0,\tfin] \times \cD 
    \right\}.
    \label{eq:reduced-manifold}
\eeq

\subsection{Deep neural networks (DNNs)}

We will define deep feed-forward neural networks . We define the set $\PP$ to contain two possible choices of activation functions in our networks. Let $\PP := \{ \sigma(x), \varsigma(x)\}$, where $\sigma(x) := \max\{0,x\}$ is the rectified linear unit (ReLU) and $\varsigma(x) := \sigma'(x)$ is the threshold function. The input variable $x$ is in $\Dom = [0,1]$ unless specified otherwise, and the output in $\RR$. Note that the inclusion of threshold functions $\varsigma$ in $\PP$ is not strictly necessary, but simplifies the exposition. On the other hand, other activations yielding universal approximations can be used without affecting the results in this work (see, e.g. \cite{Eldan16}).

We denote by $\odot$ the entry-wise composition: Given a vector of functions $\bfxi := [\xi_1, ... , \xi_N]^T$, $\xi_1, ... , \xi_N: \RR \to \RR$ and a real vector $\bfy := [y_1, ... , y_N]^T \in \RR^{N\times 1}$, the entrywise composition is given by $\bfxi \odot \bfy = [\xi_1(y_1), ... , \xi_N(y_N)]^T.$

For specified total number of layers $L \in \NN$ and the widths $\bfN = [N_1, ... , N_{L+1}] \in \NN^{L+1}$, we denote the weights, biases, and activations 
\beq
    \left\{
    \begin{aligned}
    \bfW_\ell &\in \RR^{N_{\ell+1} \times N_{\ell}}, \ell = 1, ... , L, \\
    \bfb_\ell &\in \RR^{N_{\ell+1} \times 1},
    \ell = 1, ... , L,\\
    \brho_{\ell} &\in \PP^{N_{\ell+1}},
    \ell = 1, ... , L-1,\\
    \end{aligned}
    \right.
    \quad
    \left\{
    \begin{aligned}
    \bfW &:= (\bfW_1, ... , \bfW_L),\\
    \bfB &:= (\bfb_1, ... , \bfb_L), \\
    \bfP &:= (\brho_1, ... , \brho_{L-1}).
    \end{aligned}
    \right.
    \label{eq:wgts}
\eeq
We define the corresponding set of weights, biases and activations
\beq
\left\{
\begin{aligned}
\cW(\bfN)
&=
\RR^{N_2 \times N_1}
\times \cdots \times
\RR^{N_{L+1} \times N_{L}},
\\
\cB(\bfN)
&=
\RR^{N_2 \times 1}
\times \cdots \times
\RR^{N_{L+1} \times 1},
\\
\cP(\bfN)
&=
\PP^{N_2 \times 1}
\times \cdots \times
\PP^{N_L \times 1}.
\end{aligned}
\right.
\eeq
Let us define the affine maps $A_\ell$ for $\ell = 1,2, ... , L$,
\beq
    A_\ell(\bfz) = \bfW_\ell \bfz  + \bb_\ell,
    \quad
    \bfW_\ell \in \RR^{N_{\ell+1} \times N_{\ell}},
    \quad
    \bb_\ell \in \RR^{N_\ell}.
\eeq
Entries of $\bfW_\ell$ and those of $\bfb_\ell$ are called \emph{weights} and \emph{biases}, respectively. A deep network is formed by the alternating compositions of these affine functions with activations in $\PP$.

A \emph{deep neural network (DNN)} or a \emph{deep network} with $L$ layers $\ffdP: \Dom \to \RR$ is given by
\beq
    \ffdP (x) 
    =  
    A_L \circ \brho_{L-1} \odot A_{L-1} \circ 
    ... 
    \circ \brho_2 \odot A_2 \circ \brho_1 \odot A_1,
    \label{eq:DNN}
\eeq
where $\bfW \in \cW(\bfN), \bfB \in \cB(\bfN), \bfP \in \cP(\bfN)$ for some $\bfN \in \NN^{L+1}$. We denote the class of such networks by $\cNd$, 
\beq
    \cNd 
    :=
    \left\{
        \ffdP 
        \,|\,
        \ffdP \text{ of the form \cref{eq:DNN} with }
        \bfN \in \NN^{L+1}, L \in \NN
    \right\}.
    \label{eq:cNd}
\eeq
A full deep network solution and the corresponding solution manifold is defined analogously to the full solution and the approximate solution manifold defined in \cref{sec:model-reduction}.

\begin{definition}[Full deep network solution] \label{def:M} ~
\bitme
\item Given an error threshold $\delta \in (0,1)$, if for each $u(\cdot;t,\bmu) \in \cM$ corresponding to $(t,\bmu) \in [0,\tfin] \times \cD$ there exists $\ufdP(\cdot;t,\bmu) \in \cNd$ that
\begin{itemize}
\item has dimensions $\bfN_\delta \in \NN^{L_\delta + 1} (L_\delta \in \NN)$ and the choice of activations $\bfP \in \cP(\bfN_\delta)$, both independent of $(t,\bmu)$,
\item weights $\bfW(t,\bmu) \in \cW(\bfN_\delta)$, biases $\bfB(t,\bmu) \in \cB(\bfN_\delta)$,
\item satisfies the estimate
\beq
    \Normlr{ u(\cdot ;t,\bmu) - \ufdP(\cdot;t,\bmu) }{\VV} < \delta,
    \label{eq:full-deep-error}
\eeq
\end{itemize}
then we call $\ufdP$ a \emph{full deep network solution}. 
\item We denote the \emph{full deep network solution manifold} by
\beq
\cMfd 
:= 
\left\{ \ufdP(\cdot;t,\bmu) \in \cNd 
    \,|\, 
    (t,\bmu) \in [0,\tfin] \times \cD
\right\},
\label{eq:full-deep-network-manifold}
\eeq
and say that $\cMfd$ has the dimensions $\bfN_\delta$.
\eitme
\end{definition}

\subsection{Reduced deep networks and deep reduction}

We now introduce RDNs, a deep network generalization of classical reduced models. They are derived by writing down the low-rank approximation to the weight matrices in DNNs. 

Suppose we are given a finite sample $\cNdS$ of deep networks in $\cNd$ with identical dimensions $\bfN \in \NN^{L_0 + 1}$ ($L_0 \in \NN$) and activations $\bfP \in \cP(\bfN)$. That is, 
\beq
     \cNdS
    :=
    \left\{
    \ffdPi(x) \in \cNd
    \,|\,
    i = 1, ..., S
    \right\}.
\eeq
Then let us denote the weights and biases of the $\ell$-th layer of $\ffdPi \in \cNdS$ by $\bfW_{\ell i}, \bb_{\ell i}$ for $i = 1, ... , S$ and $\ell = 1, ... , L_0$. Then we may write
\beq
    \left\{
    \begin{aligned}
    \bfW_{\ell i} 
    &= 
    \bfU_\ell \bfGamma_{\ell i} \bfV_\ell^T,
    \\
    \bb_{\ell i}
    &=
    \bfU_\ell \bc_{\ell i},
    \end{aligned}
    \right.
    \quad
    \bfU_\ell \in \RR^{N_{\ell} \times N_{\ell}}, 
    \bfV_\ell \in \RR^{N_{\ell-1} \times N_{\ell-1}},
    \bfGamma_{\ell i} \in \RR^{N_\ell \times N_{\ell-1}},
\eeq
in which $\bfU_\ell$ and $\bfV_\ell$ contain orthogonal columns.

Now, suppose that there are low-rank approximations $\tilde{\bfW}_{\ell i}$ and $\tilde{\bb}_{\ell i}$ of the form
\beq
    \tilde{\bfW}_{\ell i}
    = 
    \tilde{\bfU}_\ell 
    \tilde{\bfGamma}_{\ell i} 
    \tilde{\bfV}_\ell^T,
    \quad
    \tilde{\bb}_{\ell i} 
    = 
    \tilde{\bfU}_\ell \tilde{\bc}_{\ell i}
\eeq
in which $\tilde{\bfU}_\ell \in \RR^{N_{\ell-1} \times M_{\ell-1}}$,  $\tilde{\bfV}_\ell \in \RR^{N_\ell \times M_\ell}$, $\tilde{\bfGamma}_{\ell i} \in \RR^{M_{\ell} \times M_{\ell-1}}$,
$\tilde{\bc}_{\ell i} \in \RR^{M_\ell \times 1}$ with $M_\ell \ll N_\ell$, the columns of $\tilde{\bfU}_\ell,\tilde{\bfV}_\ell$ are columns of $\bfU_\ell,\bfV_\ell$, and $\Norm{\bfW_{\ell i} - \tilde{\bfW}_{\ell i}}{2}$ and $\Norm{\bb_{\ell i} - \tilde{\bb}_{\ell i}}{2}$ are sufficiently small. Then $A_\ell$ has a truncated version $\tilde{A}_\ell$ given by $\tilde{A}_\ell (\bfz) := \tilde{\bfW}_\ell \bfz + \tilde{\bb}_\ell.$ Projecting the input to the column space of $\tilde{\bfV}_\ell$, we obtain the reduced affine maps
\beq
    B_\ell (\bfy) := \bfGamma_\ell \bfy + \bc_\ell,
    \quad
    \bfGamma_\ell \in \RR^{M_\ell \times M_{\ell-1}},
    \bc_\ell \in \RR^{M_\ell \times 1}.
\eeq
By including a dummy input in $\bfy$ for every $B_\ell$, we may drop the bias $\tilde{\bc}_\ell$. Hence, we let without loss of generality 
\beq
    B_\ell (\bfy) := \bfGamma_\ell \bfy,
    \quad
    \bfGamma_\ell \in \RR^{M_\ell \times M_{\ell-1}}.
    \label{eq:B_ell}
\eeq
Let us define the \emph{reduced activations} $\bfxi_\ell: \RR^{M_\ell} \to \RR^{M_{\ell}}$
\beq
    \bfxi_\ell(\bfy) 
    := 
    \tilde{\bfU}_{\ell+1}^T \brho_\ell
    \odot 
    (\tilde{\bfV}_\ell \bfy),
    \quad
    (\brho_\ell \in \PP^{N_\ell}).
    \label{eq:red-activ}
\eeq
Collecting all the weights and reduced activations, let
\beq
\bfGamma := (\bfGamma_1, ... , \bfGamma_{L_0}),
\quad
\bfXi := (\bfxi_1, ... , \bfxi_{L_0-1}),
\eeq
and define the space of weights given by $\bfM = [M_1, ... , M_{L_0+1}] \in \RR^{M_{L_0+1}}$
\beq
\cG(\bfM)
=
\RR^{M_{2}\times M_1}
\times \cdots \times
\RR^{M_{L_0+1}\times M_{L_0}}.
\eeq

\begin{definition}[Reduced deep network] \label{def:RDN}
Given $\bfM \in \NN^{L+1}$,
\emph{reduced weights} $\bfGamma \in \cG(\bfM)$, and \emph{reduced activations} $\bfXi = (\bfxi_1, ... , \bfxi_{L-1})$ of the form \cref{eq:red-activ}, and $B_\ell$ of the form \cref{eq:B_ell}, we define a \emph{reduced deep network (RDN)} $\frdXi: \Dom \to \RR$ as given by
\beq
    \frdXi (x) 
    :=  
    B_L \circ \bfxi_{L-1} \odot B_{L-1} \circ ... \circ \bfxi_{2} \odot
B_{2} \circ \bfxi_{1} \odot B_1(x).
    \label{eq:RDN}
\eeq
\end{definition}

We will denote the class of reduced deep networks by, 
\beq
    \cNrd
    :=
    \left\{
        \frdXi 
        \,|\,
        \frdXi \text{ of the form \cref{eq:RDN} with }
        \bfM \in \NN^{L+1}, L \in \NN
    \right\}.
\eeq
We call the procedure of obtaining RDNs $\cNrd$ from a subset $\cNdS$ of $\cNd$ discussed above \emph{deep reduction}. The RDN is determined by the reduced activations $\boldsymbol{\Xi}$ the reduced weights $\bfGamma$, and the total number of degrees of freedom in the weight parameters is small, equal to $\sum_{\ell=1}^{L} M_\ell M_{\ell+1}$ minus the number of shared weights or biases.

The primary utility of RDN from the model reduction point of view is in finding $\urd \in \cNrd$ with small degrees of freedom such that, for each $u \in \cM$ it satisfies $\Norm{u - \urd}{\VV} < \eps$. 

\begin{definition}[Reduced deep network solution] \label{def:Mr} ~
\bitme
\item Given an error threshold $\eps \in (0,1)$, if for each $u(\cdot;t,\bmu) \in \cM$ corresponding to $(t,\bmu) \in [0,\tfin] \times \cD$ there exists $\urdXi \in \cNrd$ that
\begin{itemize}
\item has dimensions $\bfM \in \NN^{L+1}$ $(L \in \NN)$ and reduced activations $\bfXi = (\bfxi_1, ... , \bfxi_{L-1})$ of the form \cref{eq:red-activ} both independent of $(t,\bmu)$
\item has reduced weights $\bfGamma(t,\bmu) \in \cG(\bfM)$
\item satisfies the estimate
\beq
    \Normlr{u(\cdot,t;\bmu) - \urdXi(\cdot,t;\bmu)}{\VV} < \eps,
    \label{eq:rdn-accuracy}
\eeq
\end{itemize}
we call $\urdXi$ a \emph{reduced deep network solution}.
\item Denote the \emph{reduced deep network solution manifold} by
\beq
\cMrd := \left\{ \urdXi(\cdot;t,\bmu) \in \cNrd
                \,|\, (t,\bmu) \in [0,\tfin] \times \cD
         \right\}.
\eeq
and say that $\cMrd$ has the dimensions $\bfM$.
\eitme
\end{definition}

\subsection{Example: Full and reduced solutions as 2-layer networks}

As an example, we will show that classical model reduction framework from \cref{sec:model-reduction} can be expressed in terms of neural networks. A \emph{2-layer network} is a member in $\cNd$ (\cref{eq:cNd}) with two layers ($L=2$ in \cref{eq:DNN}). Such a network $f_{\brho}: \Dom \to \RR$ of width $N_{\delta} \in \mathbb{N}$ can be written in the form
\beq
    f_{\brho}(x)
    =
    \sum_{n=1}^{N_\delta} w_{2,n} \rho_n (w_{1,n} x + b_{1,n}) + b_{1,2}
    =
    \bw_2 \brho \odot ( \bw_1 x + \bb_1) + \bb_2,
    \label{eq:shallow-network}
\eeq
where $\bw_1 = [ w_{1,1}, ... , w_{1,N_\delta} ]^T \in \RR^{N_\delta \times 1}$, $\bw_2 = [ w_{2,1}, ... , w_{2,N_\delta}] \in \RR^{1 \times N_\delta}$, $\bb_1 \in \RR^{N_\delta \times 1}$, $\bb_2 \in \RR^{1 \times 1}$, $\brho = [\rho_1, ... , \rho_{N_\delta}]^T \in \PP^{N_\delta \times 1}$, $\bfW = (\bw_1, \bw_2)$, and  $\bfB = (\bb_1, \bb_2)$. 

We defined full solutions \eqref{eq:full} as piecewise linear functions on an equidistant grid with $N_{\delta}$ grid points, which can be represented as a specific 2-layer network whose weights and biases in the hidden layer is fixed. With grid-width $\Delta x := 1/(N_\delta -1)$ and the number of grid-points $N_\delta \in \NN$, set
\beq
    \begin{aligned}
    \underline{\bw}_1 
    &:=
    \frac{1}{\Delta x}\left[  1, ... , 1 \right] 
    =
    \frac{\mathbf{1}_{N_\delta}}{\Delta x},
    &
    \underline{\bb}_1 
    &:=
    [1, 0, -1, -2, ..., -N_\delta],
    \\
    \underline{\brho}
    &:=
     [ \sigma, ... , \sigma]^T,
    &
    \underline{\bb}_2 
    &:=
    \mathbf{0}.
    \end{aligned}
    \label{eq:fbar-wgts}
\eeq

Having fixed these weights and biases, only $\bw_2 = [ w_{2,1}, ... , w_{2,N_\delta}]$ is allowed to vary, so we will simplify the notation by newly denoting the variable weights $\bw_2$ by $\bw = [w_1, ... , w_{N_\delta}]$, and write
\beq
    \ffsn(x) := \bw \underline{\brho} ( \underline{\bw}_1 x + \underline{\bb}_1).
    \label{eq:fbar}
\eeq
We will denote the class of this specific networks given by \cref{eq:fbar,eq:fbar-wgts} 
\beq
    \cNs
    := 
    \left\{ 
    \ffsn (x) 
    \,|\, 
    \ffsn (x) \text{ of the form \cref{eq:fbar}}, N_\delta \in \NN
    \right\} \, .
\eeq

Then $\cNs$ is equivalent to the set of continuous piecewise linear functions on the equidistant grid: Any $\ffsn \in \cNs$ can be written as a special case of a full solution \cref{eq:full},
\beq
    \ffsn(x) = \sum_{n=1}^{N_\delta} w_n \varphi_n(x),
    \quad
    \varphi_n(x) 
    := 
    \frac{1}{\Delta x} \sigma (x - \Delta x (n-2)),
\eeq
and $\{\varphi_n\}_{n=1}^{N_\delta}$ forms a basis of the space of continuous piecewise linear functions on a equidistant grid on $\Dom = [0,1]$ of grid-width $\Delta x = 1/N_\delta$. Since $\cNs$ is dense in $\VV$, its members can serve the role of full solutions \cref{eq:full-accuracy}. Thus we can find the approximate solution manifold using the 2-layer networks in $\cNs$, and denote it by
\beq
\cMfs 
= 
\left\{ 
    \ufsn(\cdot;t,\bmu) \in \cNs 
    \,|\, (t,\bmu) \in [0,\tfin] \times \cD
\right\}.
    \label{eq:full-1layer-manifold}
\eeq
The set $\cMfs$ corresponds to the set of full solutions $\cMf$ \cref{eq:full-manifold} in classical model reduction. 

If the full 2-layer network solutions $\ufsn \in \cMfs$ \cref{eq:full-1layer-manifold} have weights $\bw(t,\bmu)$ that lie in a low-dimensional subspace with dimension $M \ll N_\delta$, then one may write
\beq
    \begin{aligned}
    \bw(t,\bmu) = \bfgamma(t,\bmu) \bfV^T,
    \quad
    \bfgamma(t,\bmu) 
    = 
    [\gamma_{1}(t,\bmu), ... , \gamma_{M}(t,\bmu)] 
    \in \RR^{1 \times M},
    \quad
    \bfV \in \RR^{N_\delta \times M},
    \end{aligned}
\eeq
in which $\bfV$ has orthogonal columns. Then one obtains the reduced representation
\beq
    \bfxi(x)
    =
    [\xi_1(x), ... , \xi_M(x)]^T
    :=
    \bfV^T
    \underline{\brho}
    \odot
    (\underline{\bw}_1 x + \underline{\bb}_1).
    \label{eq:shallow-rb}
\eeq
Each entry of $\bfxi$ is a reduced activation function \cref{eq:red-activ}. This leads to a reduced 2-layer network,
\beq
    \frsM (x)
    :=
    \bfgamma \bfxi(x) = \sum_{m=1}^{M} \gamma_m \xi_m(x).
    \label{eq:fbarr}
\eeq
We shall denote the class of such 2-layer networks 
\beq
    \cNrs
    := 
    \left\{ \frsM 
    \,|\, 
    \frsM, \bfxi \text{ of the form } 
    \cref{eq:fbarr},\cref{eq:shallow-rb},
    M \in \NN
    \right\}.
\eeq
The set of reduced solutions in $\cNrs$ that approximate the solution manifold $\cM$ form the reduced 2-layer network solution manifold,
\beq
\cMrs := \left\{ \ursM(\cdot;t,\bmu) \in \cNrs 
                        \,|\, (t,\bmu) \in [0,\tfin] \times \cD
        \right\}.
        \label{eq:cMrs}
\eeq

Then the reduced activations $\bfxi$ \cref{eq:shallow-rb} correspond to the reduced basis functions \cref{eq:reduced} in classical model reduction, and the reduced 2-layer network solution $\ursM \in \cMrs$ \cref{eq:fbarr} to a reduced solution with $M$ degrees of freedom.

\section{The Kolmogorov $N$-width of sharply convective class}

In this section, we recall the notion of Kolmogorov $N$-width and define the sharply convective class of functions. Then, we will prove a key lemma that establishes a lower bound of the Kolmogorov $N$-width of this class, showing that it decays with an algebraic rate with respect to $N$. This will be used to show the limitations of classical reduced models \cref{eq:reduced} and reduced 2-layer networks \cref{eq:fbarr}.

\subsection{Kolmogorov $N$-width}

Let us begin by defining the Kolmogorov $N$-width. Within this section, we will let $\Dom = [0,1]^d$, $d \in \NN$ since the results apply to dimensions $d > 1$, and recall that we let $\VV = L^2(\Dom)$.

\begin{definition}[\cite{pinkus12}]
\emph{The Kolmogorov $N$-width} of the set of functions $\cM$ is 
\beq
    d(N; \cM) = 
        \inf_{\VV_N} 
        \sup_{u \in \cM} \inf_{v \in \VV_N} 
        \Norm{u - v}{\VV}\,,
    \label{eq:nwidth}
\eeq
where the first infinimum is taken over all $N$-dimensional subspaces $\VV_N$ of $\VV$.
\end{definition}

When the Kolmogorov $N$-width of a solution manifold $\cM$ \cref{eq:manifold} is known, the smallest possible dimension of its reduced manifold $\cMr$ \cref{eq:reduced-manifold} that satisfies the estimate \cref{eq:reduced-accuracy} for given $\eps \in (0,1)$ is also known. This implies that classical reduced models of the form \cref{eq:fbarr} are not efficient for problems whose solution manifolds do not have a fast decaying Kolmogorov $N$-width \cite{crb-book,siamrev-survey}. For example, an exponential decay implies that an efficient classical reduced model exists, whereas an algebraic decay implies the contrary.

\subsection{Sharply convective class}

Here, we describe a key criteria we use to determine if a profile with a sharp gradient is being convected. Then we show that a set of functions satisfying this criteria have the Kolmogorov $N$-width which decays slowly with respect to $N$. 

\begin{figure}
    \centering
    \includegraphics[width=0.75\textwidth]{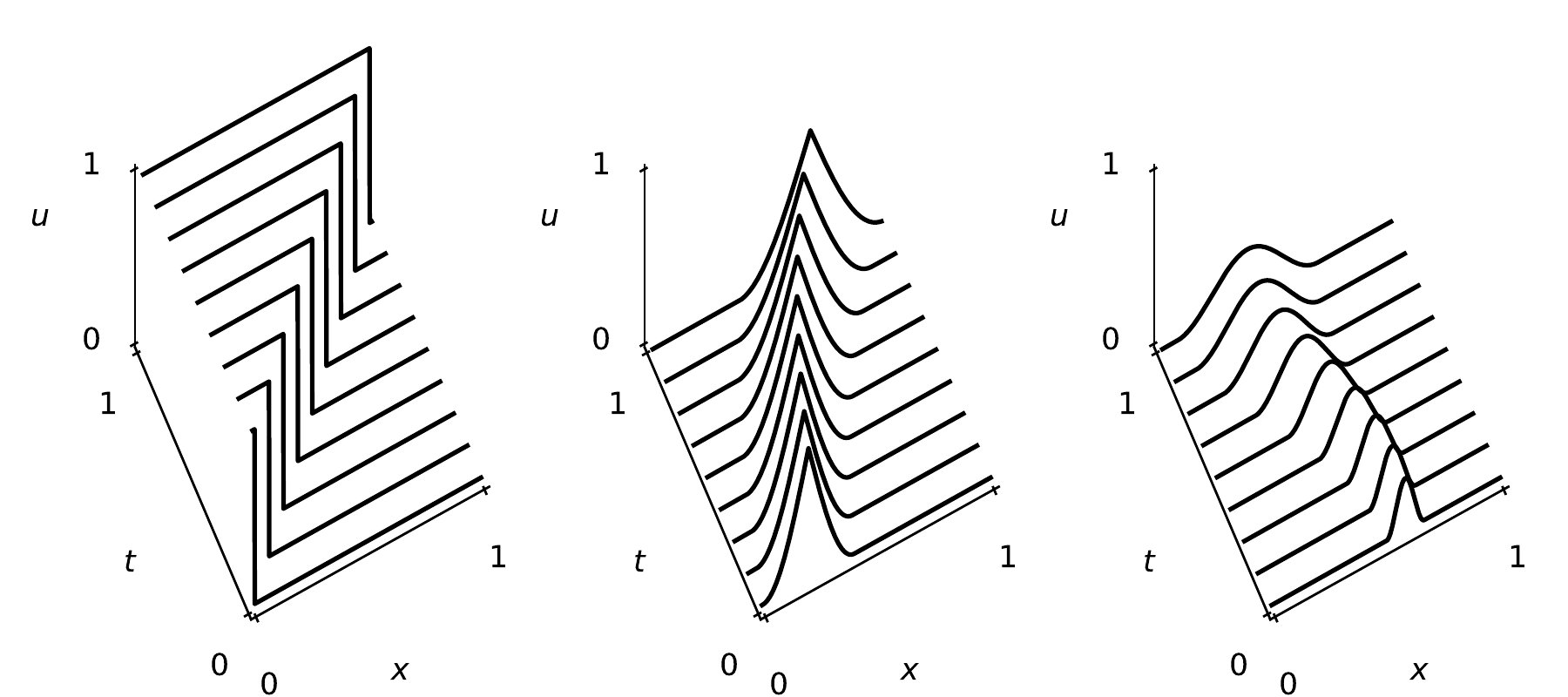}
    \caption{Examples of sharply convective classes. Solution manifolds of time-dependent problems that are $\half$-sharply convective (left),  $\frac{3}{2}$-sharply convective (middle), and $\frac{5}{2}$-sharply convective (right).}
    \label{fig:sharply-convective}
\end{figure}

\begin{definition}
A set $\cM \subset \VV$ is said to generate a $2N$-ball ($N \in \NN$) if there is a set $B_{2N}:=\{\phi_n\}_{n=1}^{2N}$ of linearly independent functions given by the sum 
\beq
    \phi_n
    =
    \sum_{k=1}^\infty a_{nk} u_{nk},
    \quad
    u_{nk} \in \cM, 
    \quad
    a_{nk} \in \RR.
    \label{eq:ball}
\eeq
For such $B_{2N}$ we will associate a real number $A_{N,p}$ given by
\beq
A_{N,p} 
    := 
    \begin{cases}
    \sup_{n} (\sum_{k=1}^\infty |a_{nk}|), &
    \Iif p = 1,\\
    \sup_{n} (\sum_{k=1}^\infty k^p|a_{nk}|^p )^{\frac{1}{p}}, &
    \Iif p \in (1,\infty).\\
    \end{cases}
    \label{eq:ANp}
\eeq
\end{definition}

We use the notation $g \lesssim h$ for real functions $g$ and $h$ to state that $g \le ch$ for some constant $c$ that does not depend on the arguments of $g$ and $h$. We also write $g \sim h$ if $g \lesssim h$ and $h \lesssim g$. We say $B_{2N} = \{\phi_n\}_{n = 1}^{2N}$ is orthogonal if the functions $\phi_1, \dots, \phi_{2N}$ are pairwise orthogonal with respect to the inner product of $\VV$.

\begin{definition}[Sharply convective class] Let $\cM \subset \VV$.
\bitme
\item $\cM$ is said to be $p$-\emph{convective} for $p \in [1,\infty)$ if it generates a $2N$-ball $B_{2N} = \{\phi_n\}_{n=1}^{2N}$, with $A_{N,p} \lesssim 1$ for all $N \in \NN$. 
\item If each $B_{2N}$ ball generated by $\Mcal$ generates an orthogonal $2N$-ball $B'_{2N} = \{\psi_n\}_{n=1}^{2N}$ with $A'_{N,p} \lesssim 1$ for certain $p \in [1,\infty)$ and $\Norm{\psi_n}{\VV} \gtrsim N^{-\alpha}$ for some $\alpha \in \RR_+$, $\cM$ is said to be $(\alpha,p)$-\emph{sharply} convective. If $\cM$ is $(\alpha,p)$-\emph{sharply convective} for all $p \in [1,\infty)$, then it is called $\alpha$-\emph{sharply convective}.
\eitme
\end{definition}

Examples of $\alpha$-sharply convective class of functions are shown in \cref{fig:sharply-convective}.

\begin{lemma}[Kolmogorov $N$-width of convective classes] \label{lem:sharp}
Let $\cM \subset \VV$.
\bitme
\item If $\cM$ is $p$-convective with the associated $2N$-ball $B_{2N}$ then $d(N; \cM) \gtrsim d(N; B_{2N})$.
\item If $\cM$ is $(\alpha,p)$-sharply convective, then $d(N; \cM) \gtrsim N^{-\alpha}$.
\eitme
\end{lemma}
\begin{proof}
Suppose $B_{2N} = \{\phi_n\}_{n=1}^{2N}$ satisfies \cref{eq:ball}. 
Then we have for any $w_{nk} \in \VV$,
\beq
    \sum_{k=1}^\infty |a_{nk}| \Norm{u_{nk} - w_{nk}}{\VV}
    \ge
    \Normlr{\sum_{k=1}^\infty a_{nk} (u_{nk} - w_{nk})}{\VV}
    = 
    \Norm{\phi_n - v_n}{\VV},
\eeq
for $v_n := \sum_{k=1}^\infty a_{nk} w_{nk}$. If $p \in (1,\infty)$ with $A_{N,p}$ as in \cref{eq:ANp}, by H\"older's inequality 
\beq
\left( \sum_{k=1}^\infty k^p \abs{ a_{nk} }^p \right)^{\frac{1}{p}}
\left( 
\sum_{k=1}^\infty \abs{ \frac{\Norm{ u_{nk} - w_{nk} }{\VV}}{k} }^q \right)^{\frac{1}{q}}
\ge
\sum_{k=1}^\infty |a_{nk}| \Norm{u_{nk} - w_{nk}}{\VV},
\eeq
where $1/p + 1/q = 1$. Then using the fact that $C_p = p^{-\frac{1}{q}} \le (\sum_{k=1}^\infty 1/k^q)^{-\frac{1}{q}}$,
\beq
    \sup_k \Norm{u_{nk} - w_{nk}}{\VV}
    \ge
    \frac{C_p}{A_{N,p}} \Norm{\phi_n - v_n}{\VV}.
\eeq
This inequality is derived similarly for the case $p=1$. Noting that $w_{nk}$ was arbitrary, for any arbitrary subspace of $N$ dimensions $\VV_N$ of $\VV$ it follows,
\[
    \begin{aligned}
    \sup_{k} \inf_{w_{nk} \in \VV_N} \Norm{u_{nk} - w_{nk}}{\VV}
    \ge
    \frac{C_p}{A_{N,p}} \Norm{\phi_n - v_n}{\VV}
    \ge
    \frac{C_p}{A_{N,p}} 
    \inf_{v \in \VV_N}
    \Norm{\phi_n - v}{\VV},
    \end{aligned}
\]
and thus
\[
    \begin{aligned}
    \sup_{u \in \cM} \inf_{w \in \VV_N} \Norm{u - w}{\VV}
    \ge
    \frac{C_p}{A_{N,p}} \inf_{v \in \VV_N} \Norm{\phi_n - v}{\VV},
    \end{aligned}
\]
because $u_{nk} \in \Mcal$. 
Since the above holds for any $\phi_n \in B_{2N}$ we take the supremum on the right-hand side,
\beq
    \sup_{u \in \cM} \inf_{w \in \VV_N} \Norm{u - w}{\VV}
    \ge
    \frac{C_p}{A_{N,p}} 
    \sup_{\phi \in B_{2N}}
    \inf_{v \in \VV_N} \Norm{\phi - v}{\VV}.
\eeq
Taking the infimum on both sides over arbitrary $N$-dimensional subspaces $\VV_N,\WW_N$ of $\VV$
\beq
    \inf_{\VV_N}
    \sup_{u \in \cM} \inf_{w \in \VV_N} \Norm{u - w}{\VV}
    \ge
    \frac{C_p}{A_{N,p}} 
    \inf_{\WW_N}
    \sup_{\phi \in B_{2N}}
    \inf_{v \in \WW_N} \Norm{\phi - v}{\VV}.
\eeq
Since $A_{N,p} \lesssim 1$ for the given $p$, this proves $d(N;\cM) \gtrsim d(N;B_{2N})$, the first part of the lemma.

Suppose each $B_{2N}$ itself generates a $2N$-ball, $B'_{2N} = \{\psi_n\}_{n=1}^{2N}$ with $A'_{N,p} \lesssim 1$. Then $d(N; B_{2N}) \gtrsim d(N; B'_{2N})$ for all $N \in \NN$. If $B'_{2N} = \{\psi_n\}_{n=1}^{2N}$ is orthogonal, we can normalize with $D_n := 1/\Norm{\psi_n}{\VV}$ and set $\hat{\psi}_n := D_n \psi_n$ with $1/D_n \gtrsim N^{-\alpha}$. Recalling that $1/A_{N,p} \ge 1/C'$ for some constant $C'$ by \cref{eq:ANp}, 
\[
    d(N;\cM)
    \ge
    \frac{C_p}{A_{N,p}} 
    \inf_{\VV_N}
    \sup_{n \in \{1, ... , 2N\}}
    \frac{1}{D_n}
    \inf_{v \in \VV_N} \Norm{\hat{\psi}_n - v}{\VV}
    \gtrsim
    \left(\frac{C_p}{\sqrt{2} C' }\right) N^{-\alpha},
\]
proving $d(N;\cM) \gtrsim N^{-\alpha}$.
\end{proof}

The rate $\alpha$ in Lemma~\ref{lem:sharp} is independent of $p \in [1,\infty)$, which means that $p$ only affects the constants.

\subsection{Example: Constant-speed transport}

Numerical experiments suggest that $\cM$ described by transport-dominated problems such as \cref{eq:pde} exhibit algebraic decay rates in their Kolmogorov $N$-width \cite{amsallem16}. The following result is a rigorous lower bound for the solution manifold of a constant coefficient advection equation.

\begin{lemma}[\cite{Ohlberger16}] \label{lem:OR}
Consider the advection equation with constant speed $\mu \in [0,1]$
\[
    u_t(x,t;\mu) + \mu u_x (x,t;\mu) = 0,
    \quad
    u(x,0;\mu) = 0, 
    \quad
    u(0,t;\mu) = 1.  
\]
Let $\cM_A := \{u(\cdot,t;1): t \in [0,1]\} \subset L^2([0,1])$ then $d_N(\cM_A)
\gtrsim N^{-\half}$.
\end{lemma}

The solution manifold is plotted in \cref{fig:sharply-convective}. The slow decay of the lower bound $N^{-\half}$ in \cref{lem:OR} implies that an efficient model reduction is provably impossible. That is, the lower bounds on the Kolmogorov $N$-width proves that, for a reduced solution $\ur$ or a reduced 2-layer network solution $\urs \in \cNrs$ to satisfy the error estimate \cref{eq:reduced-accuracy} and \cref{eq:rdn-accuracy}, respectively, for some $\eps \in (0,1)$ their dimension must be considerably large. 

Lemma~\ref{lem:OR} was proved directly in \cite{Ohlberger16}, and the $L^1$-norm version was proved in \cite{welper17} . An analogous result is proved for the wave equation in \cite{Greif19}. The ideas in the proof share commonalities with those in \cite{donoho01}. Our notion of sharply convective solution manifolds provides a concise proof.

\begin{proof}
$\cM_A$ is $\half$-sharply convective, as seen if one lets $a_{n1} = 1$, $a_{n2} = -1$, and $a_{nk} = 0$ for other $k$'s, and $u_{n1} = u_1(\cdot,\frac{n}{N+1})$, $u_{n2} = u_1(\cdot,\frac{n-1}{N+1})$. 
\end{proof}

\section{Depth separation}
We show by construction that even when the Kolmogorov $N$-width of a solution manifold decays slowly, there can exist an RDN that approximates the solution manifold with an approximation error that decreases at a geometric rate respect to the total degrees of freedom. 

Throughout, we will be concerned with a particular type of parametrized PDEs that determine the solution manifold $\cM$ \cref{eq:manifold}, namely the scalar conservation law on $\Dom := (0,1)$ that depends on the parameter $\bmu \in \cD$ of the form
\beq
    \left\{
    \begin{aligned}
        u_t + F(u,x;\bmu)_x &= \psi(u,x;\bmu), &  &\text{ in } \Dom \times (0,\tfin),\\
        u(x,0;\bmu) &= u_0(x;\bmu), &  &\text{ for } x \in \Dom,\\
        u(0,t;\bmu) &= u_0(0;\bmu). 
    \end{aligned}
    \right.
    \label{eq:pde}
\eeq
Here $F(\cdot,\cdot;\bmu) \in C^\infty(\RR \times \Dom)$ and $F(\cdot,x;\bmu)$ is strictly convex. For theoretical and numerical results regarding the PDE, we refer the reader to standard references \cite{lax72,evans10,fvmbook}.

\subsection{MATS architecture}

We introduce a special architecture to be used in our constructive RDN approximations, by extending the MATS approximation \cite{Rim19}. A specialized component is the implementation of the inverse using the network architecture, so we briefly discuss how we can approximate the inverse of a monotonically increasing function, using a deep network. In particular, we focus on the case the given function is itself also a network in $\cNd$. 

\begin{lemma} \label{lem:inv}
Given non-constant $f \in \cNd$ defined on $\Dom = [0,1]$  with $L$ layers satisfying $f' \ge 0$ in the sense of distributions, there is an approximate inverse $f^\flat \in \cNd$ which has $(L+4) \Linv$ layers such that $\Norm{f^{-1} - f^\flat}{L^\infty(\Dom)} \le \abs{\Dom} 2^{-\Linv}$. We will call $f^\flat$ \emph{approximate inverse with $\Linv$ layers}.
\end{lemma}

\begin{proof}
The proof implements the bisection algorithm and is given in \cref{sec:proof:lem:inv}. 
\end{proof}

In the case where $f \in \cNd$ has a jump of positive magnitude, the bisection algorithm converges to the location of of the jump. The construction is easily extended to multiple spatial dimensions, although we will not make use of such extensions here.

Let us define,
\beq
    \begin{aligned}
    \cNd_\sharp &:= \{f \in \cNd \,|\, f' \ge 0 \},\\
    \cNd_\flat &:= \{f^\flat \,|\, f \in \cNd_\sharp \},
    \end{aligned}
    \quad
    \begin{aligned}
    \cNrd_\sharp
    &:=
    \{ f^{(r)} \in \cNrd \,|\,
    (f^{(r)})' \ge 0 \},    
    \\
    \cNrd_\flat
    &:=
    \{ g^\flat \,|\,
    g \in \cNrd_\sharp  \}.
    \end{aligned}
\eeq
Now, we specify the architecture we will use in the remainder of the work,
\beq
    \begin{aligned}
    \cNmats
    &:= 
    \{v \circ T_{L-1} \circ ... \circ T_1
    \,|\,
    v \in \cNd,
    T_1, ...  ,T_{L-1} 
    \in \cNd_\sharp
        \cup
        \cNd_\flat, L \in \NN
    \},
    \\
    \cNrmats
    &:= 
    \{v \circ T_{L-1} \circ ... \circ T_1
    \,|\,
    v \in \cNrd,
    T_1, ...  ,T_{L-1} 
    \in \cNrd_\sharp
        \cup
        \cNrd_\flat, L \in \NN
    \}.
    \end{aligned}
\eeq

\subsection{The regular and linear case: the color equation}
\label{sec:color}

Let us begin by defining two classes of functions.
Let $\mathbb{U}$ denote piecewise analytic functions. That is, $u \in \UU$ if and only if there exists a finite partition $\breve{\Omega}$ of $\Dom$ by intervals so that $u |_{\Omega'}$ is analytic for each $\Omega' \in \breve{\Dom}$.
Then, let $\mathbb{T} := \{T : T \in \mathbb{U}, T' \ge 0, T' \not\equiv 0\}$ where the derivative is in the sense of distributions.

Note that if $T \in \mathbb{T}$, $T^{-1}$ is well-defined near $T(x_0)$ for any $x_0 \in \Dom$ if $T'(x_0) > 0$. Since the set $\{T(x) : x \in \Dom, \exists y \in \Dom, x \ne y \text{ s.t } T(x) = T(y)\}$ has zero measure, $T^{-1}$ is defined almost everywhere in $T(\Dom)$, so for any $v \in \mathbb{U}$, $v \circ T^{-1} \in \mathbb{U}$ is defined almost everywhere.

\subsubsection{Limitation of classical reduced models}

Consider the solutions to the color equation (variable-speed transport) for which $F(u,x;\bmu) = c(x;\bmu) u$ and $\psi(u,x;\bmu) = c_x(x;\bmu) u$ in \cref{eq:pde}, with a given set of initial conditions $\UU_0 := \{u_0(\cdot;\bmu):\bmu \in \cD\} \subset \UU$.

\begin{assumption}\label{as:CC}
Let us assume that
\bitme
\item for each $u_0(\cdot;\bmu) \in \UU_0$, the maximal interval $\Dom' \subset \Dom$ containing $x \in \Dom$ on which it is analytic satisfies, $|\Dom'| > \gamma |\Dom|$ for some constant $\gamma > 0$,
\item the Kolmogorov $N$-width of $\UU_0$ decays exponentially with respect to $N$,
\item $\Norm{u_0'(\cdot;\bmu)}{L^\infty(\Dom)}, \Norm{u_0(\cdot,\bmu)}{TV(\Dom)} \lesssim 1$ where $\Norm{\cdot}{L^\infty(\Dom)}$ is the essential supremum  and $\Norm{\cdot}{TV(\Dom)}$ is the total variation, 
\item for all $\bmu \in \cD$, $c(w;\bmu)$ is analytic in $R:=\{w \in \CC \,|\, |w - x_0| < b, x_0 \in \Dom\}$ for some $b > b_0 > 0$, and $0 < c_0 \le c(x;\bmu) \lesssim 1$,
\item final time $\tfin < b/\sup_{w \in R, \bmu \in \cD} |c(w;\bmu)|$.
\eitme
\end{assumption}

We will denote by $\cM_C$ the solution manifold of such a parametrized PDE.

One can solve for each solution in $\cM_C$ by the method of characteristics by integrating along the characteristic curves \cite{evans10}. We will denote the characteristic curve for the initial condition $x_0$ by $X(t;x_0,\bmu)$. Then the ODEs for the characteristic curves are
\beq
\left\{\begin{aligned}
    X'(t;x_0,\bmu) &= c(X(t;x_0);\bmu), \quad t \in (0,\tfin),\\
     X(0;x_0,\bmu) &= x_0.
\end{aligned}\right.
    \label{eq:ode}
\eeq 
By classical ODE theory \cite[Theorem 8.1]{coddlev55}, $X(t;x_0)$ ($x_0 \in \Dom$) is analytic with respect to the variable $t$ in the neighborhood of $(0,\tfin)$.  We will write $X$ also as a function of its initial condition, $X(t,x;\bmu) := X(t;x,\bmu)$. Since $c$ is bounded away from zero, $\partial_{x} X > 0$ for $t \in (0,\tfin)$ ensuring that the the map is strictly increasing function of $x$. Furthermore, the following lemma shows that $X$ is analytic with respect to $x$.

\begin{figure}
    \centering
    \includegraphics[width=0.75\textwidth]{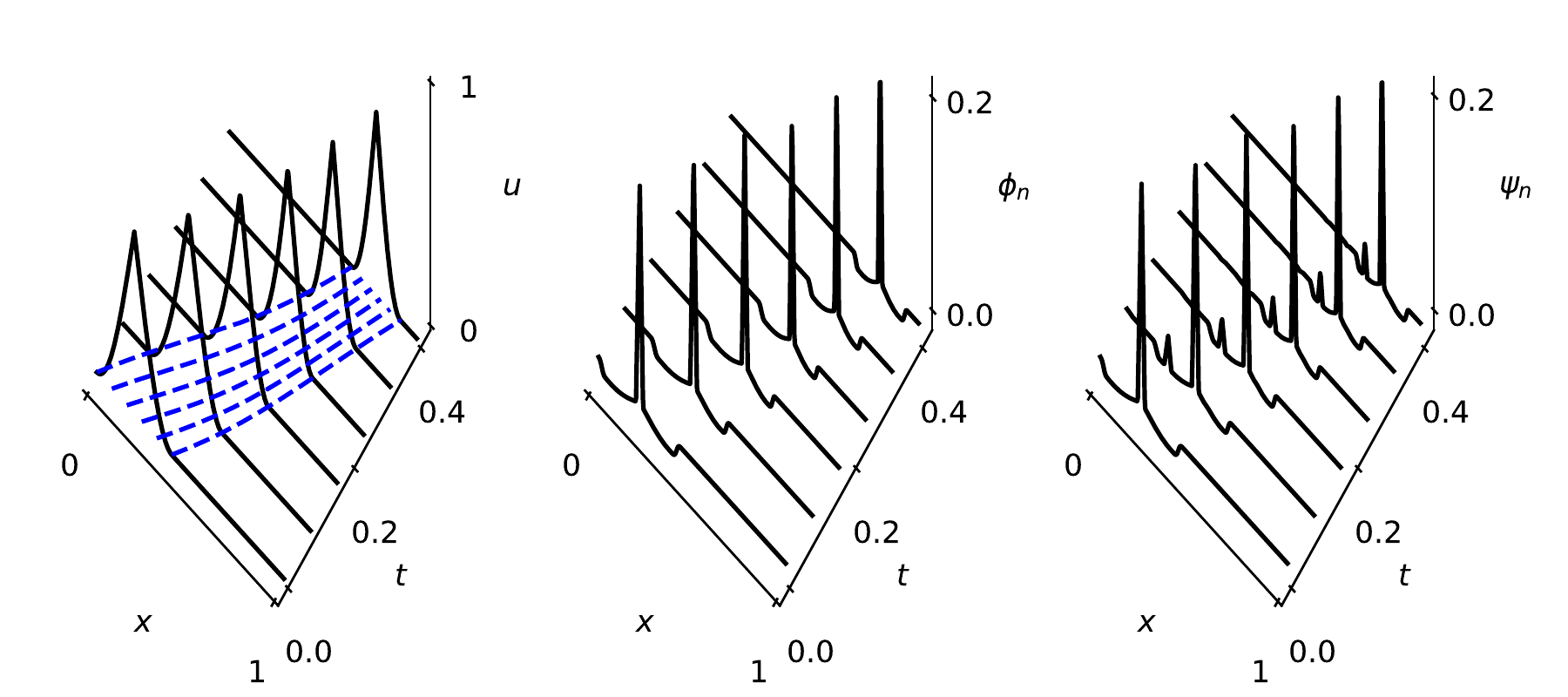}
    \caption{A solution manifold of the color equation with speed \cref{eq:speed-expl} with parameters values $\bmu  = (0.3,2\pi, \pi)$, along with characteristic curves $X(\cdot,x_0;\bmu)$ \cref{eq:ode} (left), an example of a $2N$-ball $\{\phi_n\}_{n=1}^{2N}$ derived from a finite difference stencil \cref{eq:vandermonde} (middle), an orthogonalized $2N$-ball $\{\psi_n\}_{n=1}^{2N}$ obtained by the Gram-Schmidt process \cref{eq:gram-schmidt} (right).}
    \label{fig:color-curves}
\end{figure}

\begin{lemma}\label{lem:analytic}
The characteristics $X(t,x;\bmu)$ \cref{eq:ode} in which $c$ satisfies the conditions in \cref{as:CC}, is analytic in $x \in \Dom$.
\end{lemma}

\begin{proof}
The proof is straightforward and is given in \cref{sec:proof:lem:analytic}.
\end{proof}

We express the transformation of the domain by
\beq
    T_{(t,\bmu)} : \Dom \to X(t,\Dom;\bmu) 
    \quad \text{ given by } \quad
    T_{(t,\bmu)}(x) := X(t,x;\bmu)
    \label{eq:Ttmu}
\eeq
so that $u_0(T_{(t,\bmu)}^{-1}(x)) = u(x,t;\bmu)$, $u_0(x) = u(T_{(t,\bmu)}(x),t;\bmu)$. Since $T_{(t,\bmu)}' > 0$, by virtue of \cref{lem:analytic} $T_{(t,\bmu)} \in \TT$.

\begin{example}
Consider the variable speed,
\beq
    c(x;\bmu)
    =
    1.5 + \mu_1 \sin(\mu_2 x) + 0.1 \cos(\mu_3 x),
    \label{eq:speed-expl}
\eeq
where $\bmu = (\mu_1, \mu_2, \mu_3) \in [0.25,0.50] \times [2\pi, 6\pi] \times [\pi, 1.1\pi]$. The the characteristic curves $X(t,x_0)$ and the corresponding transport maps $X(t_0,x)$ \cref{eq:ode} for two different values of $\bmu$ are shown in \cref{fig:color-curves}. For a numerical experiment concerning this example, see \cite{Rim19}.
\end{example}

Now we prove the lower bound for the Kolmogorov $N$-width of $\cM_C$ in the special case the solution is not analytic over the entire domain $\Dom$.

\begin{theorem} \label{thm:colorlb}
If certain $u_0 \in \UU_0$ is at most $s$-times continuously differentiable, that is, there is $s \in \NN_0$ for which $u_0^{(s)} \in C(\Dom)$ but $u_0^{(s+1)} \notin C(\Dom)$, then $d(N;\cM_C) \gtrsim N^{-s-\half}$. 
\end{theorem}

\begin{proof}
The lower bound holds for the solutions of the PDE for a single fixed parameter. Let us fix $\bmu^* \in \cD$, and $u_0 \in \UU_0$ be chosen to satisfy the hypothesis, then let
\beq
    \cM_C^* := \{ u(\cdot,t;\bmu^*) : u(\cdot,0;\bmu^*) = u_0, t \in [0, \tfin]\}.
\eeq
Then $\cM_C^* \subset \cM_C$ so $d(N;\cM_C) \ge d(N;\cM_C^*)$. Let us introduce the shorthands $T(x,t) := T_{(t,\bmu^*)}(x)$, and $T^\dagger(x,t) := T^{-1}_{(t,\bmu^*)}(x)$.

Recall that the maximal intervals $\Dom_\ell \subset \Dom$ in which $u_0$ is analytic satisfy $|\Dom_\ell| \ge \gamma |\Dom|$. Then there exists a $s \in \NN$ for which $u_0^{(s+1)}$ has a jump discontinuity at $\inf \Dom_\ell$ for some $\ell$. Set $x_0 \in \Dom$ as the location of the jump. By the chain rule,
\beq
\p_t^{(s+1)} u(x,t) 
=
\p_t^{(s+1)} u_0(T^*(x,t)) 
= 
u_0^{(s+1)}(T^\dagger(x,t)) \p^{(s+1)}_t (T^\dagger(x,t)) + \cdots, 
\label{eq:ptsu}
\eeq
so $\p_t^{(s+1)} u(\cdot,t)$ has a jump at $T(x_0,t)$ as well. 

We further restrict the functions in $\cM_C^*$ to $\Dom_\ell$, letting $\cM_C^{**} := \{u |_{\Dom_\ell}: u \in \cM_C^*\}$ then $d(N;\cM_C) \ge d(N;\cM_C^{**})$.

We will make use of a finite difference approximation to $\p_t^{s_1} u$ with $s_1 > s$. Let us fix $N$, and choose a (scaled) finite difference stencil of size $K$, $( b_k )_{k = 1}^K$ over an equidistant grid near  $\tau_n \in[0,t_1]$, $t_1 \in (0,\tfin)$ being the finite time during which $u(\cdot,t)$ is analytic in $\Dom_\ell \setminus \{T(x_0,t)\}$ for $t \in [0,t_1]$. We denote the equidistant grid by
\beq
    \tau_{nk} := \tau_n + (k-1) (\Delta t),
    \qquad
    n = 1, ... , 2N, 
    \quad
    k = 1, ... , K.
    \label{eq:taunk}
\eeq
To approximate the $\p_t^{s_1} u$ with $K$ orders of accuracy, we use the scaled stencil given by the Vandermonde system, which solves for $( b_k )_{k=1}^K$ with $K > s_1 > s$ and $i=1, ... , K$
\beq
    \begin{aligned}
    \sum_{k=1}^K \frac{(\tau_{nk} - \tau_n)^{i-1}}{(i-1)!} b_k 
    =
    \sum_{k=1}^K \frac{((k-1) \Delta t)^{i-1}}{(i-1)!} b_k 
    =
    \begin{cases}
    (\Delta t)^{s_1},& \text{ if } i-1 = s_1\\
               0,& \text{ otherwise,} \\
    \end{cases}
    \end{aligned}
\eeq
(see for example \cite[Ch2]{fdmbook}). Dividing through the $i$-th equation by $(\Delta t)^{i-1}$ yields 
\beq
    \sum_{k=1}^K \frac{(k-1)^{i-1}}{(i-1)!} b_k 
    =
    \begin{cases}
        1,& \text{ if } i-1 = s_1,\\
        0,& \text{ otherwise,} \\
    \end{cases}
    \label{eq:vandermonde}
\eeq
so each $b_k$ is independent of $\Delta t$. We choose $\{\tau_n\}_{n=1}^{2N}$ and the stencil width $\Delta t$ small enough so that $\Delta t \le \half$, $\Delta t \lesssim N^{-1}$ and 
\beq
    S_n := \{x \in \Dom_\ell \,|\, |x - T(x_0,\tau_n)| \le \bar{c} K\Delta t \},
    \quad
    \bar{c} := \sup_{x \in \Dom} c(x;\mu^*),
    \label{eq:Sn}
\eeq
are pairwise disjoint for $n = 1, \dots, 2N$. 

Take $\phi_n := \sum_{k=1}^K b_k u_{nk}$, in which $u_{nk} := u(\cdot, \tau_{nk}).$  Then since $u_{nk}$ is analytic in $\Dom_\ell \setminus S_n$,
\beq
   \abs{\phi_n(x)
    - 
    (\Delta t)^{s_1}
    \p_t^{s_1} u(x,\tau_n)} 
    \lesssim (\Delta t)^{K}
    \,\text{ in }
    x \in \Dom_\ell \setminus S_n,
    \text{ for } n = 1, ... , 2N,
\eeq
as can be derived from the the relation \cref{eq:vandermonde}. Now $u(\cdot,\tau_n)$ is analytic in $\Dom_\ell \setminus \{T(x_0,\tau_n)\}$ therefore $\abs{\phi_n} \lesssim (\Delta t)^{s_1} $ in $\Dom_\ell \setminus S_n$ because for $x \in \Dom_\ell \setminus S_n$,
\beq
    \abs{\phi_n(x)}
    \le
    \abs{\phi_n(x) - (\Delta t)^{s_1}
    \p_t^{s_1} u(x,\tau_n)}
    +
    \abs{(\Delta t)^{s_1}
    \p_t^{s_1} u(x,\tau_n)}
    \lesssim (\Delta t)^{s_1}.
\eeq
In contrast, in $S_n$ that contains the location of singularities of $\phi_n$, finite-difference error estimates (\cite{fdmbook}) yield $|\phi_n| \sim (\Delta t)^s$.

Noting that $S_n$ \cref{eq:Sn} has the measure $\abs{S_n} = \bar{c} K (\Delta t)$, 
\beq
    \begin{aligned}
    \NormLzdl{\phi_n}^2
    = \left(\int_{\Dom_\ell \setminus S_n}  + \int_{S_n} \right) \abs{\phi_n(x)}^2 \dx
    \lesssim  \abs{\Dom_\ell}(\Delta t)^{2s_1} + \abs{S_n} (\Delta t)^{2s}
    \lesssim (\Delta t)^{2s+1}.
    \end{aligned}
\eeq
A similar calculation yields the lower bound $\NormLzdl{\phi_n}^2 \gtrsim (\Delta t)^{2s + 1}$, so $\NormLzdl{\phi_n} \sim (\Delta t)^{s + \half}$, and furthermore
\beq
    \abs{(\phi_n,\phi_m)_{\Dom_\ell \setminus (S_n \cup S_m)}}
    \lesssim  (\Delta t)^{2s_1},
    \quad
    \abs{(\phi_n,\phi_m)_{S_n \cup S_m}}
    \lesssim
    (\Delta t)^{s_1 + s + 1}.
\eeq
If we let $\hat{\phi}_n := \phi_n / \NormLzdl{\phi_n}$, then for $n \neq m$,
\beq
    \begin{aligned}
    \abs{(\hat{\phi}_n,\hat{\phi}_m)} 
    &= 
    \frac{\abs{(\phi_n,\phi_m)}}
    { \NormLzdl{\phi_n} \NormLzdl{\phi_m}}
    \lesssim 
    (\Delta t)^{s_1 -2s-1},
    \end{aligned}
\eeq
in which $C_7(K,s_1) = (\half C_5(K,s_1) + C_6(K,s_1)) / c_4(K,s_1)$.
\beq
    \sum_{m \ne n} \abs{ (\hat{\phi}_n, \hat{\phi}_m)}
    \lesssim 2N (\Delta t)^{s_1 - 2s - 1}
    \lesssim (\Delta t)^{s_1 - 2s - 2}.
\eeq
So for any $s_1 > 2s + 2$, if $\Delta t$ is smaller than a constant $C_1$ that depends only on $s_1$ and $K$, the matrix $\bfC_{nm} := (\hat{\phi}_n, \hat{\phi}_m)$ is strictly diagonally dominant and thus invertible \cite[Corollary 5.6.17]{horn12}, and so $\{\phi_n\}$ is linearly independent. Letting $B_{2N} = \{\phi_n\}_{n=1}^{2N}$, it follows that $\cM^{**}_C$ is $p$-convective, and by the first part of \cref{lem:sharp}, $d(N; \cM_C^{**}) \gtrsim d(N; B_{2N})$.

We now show that each $B_{2N}$ generates an orthogonal $2N$-ball $B'_{2N}$ with $p=1$ and $A'_{N,p} \lesssim 1$. By the Gram-Schmidt process (without normalization), define $\{\psi_n\}_{n=1}^{2N}$,
\beq
    \psi_n 
    := 
    \phi_n - \sum_{m < n} (\phi_n, \hat{\psi}_m) \hat{\psi}_m
    = 
    \phi_n + \sum_{m < n} \theta_{nm} \psi_m,
    \label{eq:gram-schmidt}
\eeq
in which $\hat{\psi}_m := \psi_m / \NormLzdl{\psi_m}$, and $\theta_{nm} :=- (\phi_n,\psi_m)/\NormLzdl{\psi_m}^2$. Then writing \cref{eq:gram-schmidt} as a linear system,
\beq
\left[
\phi_1, \phi_2, ... , \phi_{2N} 
\right] 
=
\left[
\psi_1, \psi_2, ... , \psi_{2N} 
\right] \bfTheta,
\quad
\bfTheta :=
\begin{bmatrix}
1 & \theta_{21} &  \theta_{31} & \cdots & \theta_{(2N)1} \\
  &           1 &  \theta_{32} & \cdots & \theta_{(2N)2} \\
  &             &           1  & \ddots &         \vdots \\
  &             &              & \ddots &   \theta_{(2N)(2N-1)} \\
  &             &              &        &              1 \\
\end{bmatrix}.
\label{eq:gram-schmidt-triu}
\eeq
We claim that $\abs{\theta_{nm}} \lesssim (\Delta t)^{s_1 -2s -\frac{3}{2}}$. Firstly,
\beq
\begin{aligned}
\frac{\abs{(\phi_n,\psi_m)}}{\NormLzdl{\psi_m}}
&\le 
\sum_{m < n} 
\frac{\abs{(\phi_n, \phi_m)}}{\NormLzdl{\phi_m}}
\lesssim
\frac{ (\Delta t)^{s_1} (2N)^\half}{\min_{p \le m} \NormLzdl{\phi_p}} 
\lesssim 
(\Delta t)^{s_1 - s - 1},
\end{aligned}
\eeq
Secondly,
\beq
    \begin{aligned}
    \NormLzdl{\phi_n} - \NormLzdl{\psi_n} 
    &=
    \NormLzdl{\phi_n}
    - \NormLzdl{\phi_n - \sum_{m < n} (\phi_n, \hat{\psi}_m) \hat{\psi}_m}    
    \le
    \sum_{m < n} \abs{(\phi_n, \hat{\psi}_m)} 
    \\&\le
    \sum_{m < n} \abs{(\phi_n, \hat{\phi}_m)} 
    \le
    \sum_{m < n} \abs{(\hat{\phi}_n, \hat{\phi}_m)}
    \NormLzdl{\phi_n}
    \\&\lesssim
    2N (\Delta t)^{s_1 -2s -1} 
    \NormLzdl{\phi_n}
    \lesssim  (\Delta t)^{s_1 -2s -2} \NormLzdl{\phi_n}.
    \end{aligned}
\eeq
Since we will choose $s_1 > 2s+ 2$, for $\Delta t$ smaller than a constant $C_2$ that depends only on $s_1$ and $K$, we have $\NormLzdl{\phi_n} \lesssim \NormLzdl{\psi_n}$ and so $\NormLzdl{\psi_n} \gtrsim (\Delta t)^{s + \half}$. This proves the upper bound on $\abs{\theta_{mn}}$.

Now let $\psi_n := \sum_{k=1}^n a_{nk} \phi_k$ for $n = 1, ... , 2N$ in which $a_{nk}$ is the $(k,n)$-th entry of $\bfTheta^{-1}$ in which $\bfTheta$ was defined in \cref{eq:gram-schmidt-triu}. The strictly upper triangular entries of $\bfTheta^{-1}$ must scale equivalently with those of $\bfTheta$. Thus we have $\abs{a_{nk}} \lesssim (\Delta t)^{s_1-2s-\frac{3}{2}}$ for $k < n$ and $a_{nn} = 1$ ($n = 1, ... , 2N$). So for any $s_1 > 2s + 2$ as above,
\beq
    \sum_{k=1}^n \abs{a_{nk}}
    \lesssim
    1 + \sum_{k=1}^{2N} \frac{1}{N^{s_1-2s-\frac{3}{2}}}
    \lesssim 1.
\eeq    

We let $\Delta t \sim N^{-1}$, where $\Delta t$ is ensured to be smaller that constants $C_1, C_2$ that depend on $s_1$ and $K$ defined above. Then we have shown that $B_{2N}$ itself generates the $2N$-ball $B'_{2N} = \{\psi_n\}_{n=1}^{2N}$, with $A'_{N,1} \lesssim 1$. Since $B'_{2N}$ is orthogonal and $\NormLzdl{\psi_n} \gtrsim \NormLzdl{\phi_n} \gtrsim (\Delta t)^{s+\half} \sim N^{-s-\half}$, we obtain that $d(N;\cM_C) \gtrsim N^{-s-\half}$ by applying the second part of \cref{lem:sharp}.
\end{proof}

\subsubsection{An efficient RDN approximation}

We will first approximate $T_{(t,\bmu)}$ using Chebyshev polynomials,
as its analyticity (\cref{lem:analytic}) implies that its polynomial approximation will converge geometrically. The convergence rate of the Chebyshev basis to an analytic function is expressed in terms of the radii of Bernstein ellipses \cite{trefethen}.

\begin{assumption}\label{as:BE}
Let $\rho(t,\bmu)$ denote the radius of the closed Bernstein ellipse $\bar{E}_\rho(t,\bmu)$ in which $T_{(t,\bmu)}$ is analytic.
Suppose that
\beq
    \rho_* :=  \inf_{t,\bmu} \rho(t,\bmu) > 1,
    \Aand
    \sup_{x,t,\bmu} |T_{(t,\bmu)}(x)| < \infty.
    \label{eq:bernstein-radius}
\eeq
\end{assumption}

\begin{lemma}\label{lem:chebyshev}
If \cref{as:BE} holds, then we have
\beq
    \Norm{T_{(t,\bmu)} - \mathfrak{T}^M_{(t,\bmu)}}{L^\infty(\Dom)} 
    \lesssim
    \rho_*^{-M},
    \quad
    \Norm{T'_{(t,\bmu)} - (\mathfrak{T}^M_{(t,\bmu)})'}{L^\infty(\Dom)} 
    \lesssim 
    \rho_*^{-M}.
    \label{eq:chebyshev-error}
\eeq
in which $\mathfrak{T}^M_{(t,\bmu)}$ is a Chebyshev polynomial of degree $M$ which is expressed in terms of the Chebyshev basis $\{p_m\}_{m=1}^\infty$ as 
\beq
    \mathfrak{T}^M_{(t,\bmu)}(x) = \sum_{m=1}^M \gamma_m(t,\bmu) p_m(x).
\eeq
\end{lemma}

\begin{proof}
Follows directly from \cref{lem:analytic} and \cite[Theorem 8.6]{trefethen}.
\end{proof}

Since Chebyshev polynomials are smooth, they can be approximated by a 2-layer network $\underline{\xi}_m \in \cNs$ of width $N_\delta$ that satisfies $\Norm{p_m - \underline{\xi}_m}{L^\infty(\Dom)} < \eps$, $\Norm{p'_m - \underline{\xi}'_m}{L^\infty(\Dom)} < \eps$ for any given $\eps \in (0,1)$. This yields the approximation of $T^M_{(t,\bmu)}(x)$ in $\cNrs$ of the form
\beq
    T^M_{(t,\bmu)}(x) = \sum_{m=1}^M \gamma_m(t,\bmu) \underline{\xi}_m(x)
    \label{eq:1layer-chebyshev}
\eeq
that satisfies the same estimates as \cref{eq:chebyshev-error}
\beq
    \Normlr{T_{(t,\bmu)} - T^M_{(t,\bmu)}}{L^\infty(\Dom)} 
    \lesssim
    \rho_*^{-M},
    \quad
    \Normlr{T'_{(t,\bmu)} - (T^M_{(t,\bmu)})'}{L^\infty(\Dom)} 
    \lesssim 
    \rho_*^{-M}.
\eeq
Due to the the derivatives being uniformly accurate, and since $T_{(t,\bmu)}' \ge c_0  / c_1 > 0$ in which $c_1 = \sup_{(x,\bmu) \in \Dom \times \cD} c(x;\bmu)$, we have that $(T_{(t,\bmu)}^M)' > c_0/c_1 - \eps$ for $M \gtrsim \abs{\log \eps}$. 

Next, we construct an RDN approximation using $\{\underline{\xi}_m\}_{m=1}^M$ as reduced activations in the hidden layer.

\begin{theorem} \label{thm:construct}
Suppose the solution manifold $\cM_C$ satisfying \cref{as:CC} also has characteristic curves $T_{(t,\bmu)}$ that satisfy \cref{as:BE}. Then for any error threshold $\eps \in (0,1)$ there exists a reduced deep network solution manifold $\cMrd_C \subset \cNrmats$ with total degrees of freedom $M \lesssim |\log \eps|$. 
\end{theorem}

\begin{proof}
In this proof, we will fix $(t,\bmu)$ and denote $T = T_{(t,\bmu)}$, $T_{M_1} = T^{M_1}_{(t,\bmu)}$, $T_{M_1}^\flat = (T_{(t,\bmu)}^{M_1})^\flat$ and $T_{M_1}^{-1} = (T_{(t,\bmu)}^{M_1})^{-1}$, where $T_{(t,\bmu)}^{M_1}$ is the approximation \cref{eq:1layer-chebyshev} in $\cNs$. 

Let us choose a reduced 2-layer network $\uosl{M_2} \in \cNrs$ so that for all $\bmu \in \cD$ it satisfies,
\beq
    \Normlr{ u_0(\cdot;\bmu) - \uosl{M_2}(\cdot;\bmu) }{\VV} < \rho_*^{-M_1},
\eeq
in which $\rho_* > 1$ is the lower bound on the Bernstein radii as denoted in \cref{eq:bernstein-radius}, and $M_1$ is chosen to satisfy
\beq
    \Normlr{T - T_{M_1}}{\VV} 
    \lesssim
    \rho_*^{-M_1}.
\eeq
By \cref{as:CC}, the Kolmogorov $N$-width of $\UU_0$ decays exponentially, so $M_2$ can be chosen to satisfy $M_2 \sim M_1$.

We propose the following reduced deep network solution in $\cNmats$, 
\beq
    \urd(x) := \uosl{M_2} \circ T_{M_1}^\flat(x).
\eeq
The approximation satisfies,
\beq
    \Normlr{ u_0 \circ T^{-1} - \uosl{M_2} \circ T_{M_1}^{\flat} }{\VV}
    \le 
    \Normlr{u_0 \circ T^{-1} - u_0 \circ T_{M_1}^{\flat} }{\VV}
       + 
    \Normlr{u_0 \circ T^{\flat}_{M_1} - \uosl{M_2} \circ T_{M_1}^{\flat} }{\VV}.
\eeq
We will use the approximate inverse with $\Linv$ layers for $T_{M_1}^\flat$ (\cref{lem:inv}), for which we choose $\Linv$ sufficiently large so that $2^{-\Linv} < \rho_*^{-{M_1}}$ and $\Linv \sim M_1$. We bound the first term on the right,
\beq
    \begin{aligned}
    &\Normlr{u_0 \circ T^{-1} - u_0 \circ T_{M_1}^{\flat} }{\VV}
    \\&\le
    \Normlr{u_0 \circ T^{-1} - u_0 \circ T_{M_1}^{-1} }{\VV}
    +
    \Normlr{u_0 \circ T_{M_1}^{-1} - u_0 \circ T_{M_1}^{\flat} }{\VV}
    \\&\le
    \left(\Normlr{u_0'}{L^\infty(\Dom)} + \Normlr{u_0}{TV(\Dom)} \right)
    \Normlr{T'}{L^\infty(\Dom)}^\half
    \Normlr{(T_{M_1}^{-1})'}{L^\infty(\Dom)}
    \Normlr{T - T_{M_1}}{\VV}
    \\&\qquad + 
    \Normlr{u_0'}{L^\infty(\Dom)}
    \Normlr{T_{M_1}^{-1} - T_{M_1}^{\flat}}{\VV}
    \\&\lesssim
    \left( \Normlr{u_0'}{L^\infty(\Dom)} + \Normlr{u_0}{TV(\Dom)} \right)
    \Normlr{T'}{L^\infty(\Dom)}^\half
    \Normlr{(T_{M_1}^{-1})'}{L^\infty(\Dom)}
    \rho_*^{-{M_1}}
    +
    |\Dom|
    \Normlr{u_0'}{L^\infty(\Dom)}
    2^{-\Linv},
    \end{aligned}
\eeq
and for the second term
\beq
   \begin{aligned} 
      \Normlr{ u_0 \circ T_{M_1}^\flat  - \uosl{M_2} \circ T_{M_1}^\flat }{\VV}
      &\le 
      \Normlr{T_{M_1}'}{L^\infty(\Dom)}^\half
      \Normlr{u_0 \circ T_{M_1}^\flat \circ T_{M_1}
            - \uosl{M_2} \circ T_{M_1}^\flat \circ T_{M_1} }{\VV}
      \\&\le
      \Normlr{T_{M_1}'}{L^\infty(\Dom)}^\half
      \left(
        \Normlr{u_0  - \uosl{M_2}}{\VV} + \eps_1
      \right)
   \end{aligned} 
\eeq
in which $\eps_1 \lesssim 2^{-\Linv}$ since $T_{M_1}^\flat \circ T_{M_1}$ is a piecewise constant approximation of the identity on a grid with $2^{\Linv}$ grid-points.

Then putting it together,
\beq
    \begin{aligned}
    &\Normlr{ u_0 \circ T^{-1} - \uosl{M_2} \circ T_{M_2}^{\flat} }{\VV}
    \\
    &\lesssim 
    \left(\Normlr{u_0'}{L^\infty(\Dom)} + \Normlr{u_0}{TV(\Dom)} \right)
    \Normlr{T'}{L^\infty(\Dom)}^\half
    \Normlr{(T_{M_1}^{-1})'}{L^\infty(\Dom)}
     \rho_*^{-{M_1}}
    +
    |\Dom|
    \Normlr{u_0'}{L^\infty(\Dom)}
    2^{-\Linv}
    \\&\qquad +
    \Normlr{T_{M_1}'}{L^\infty(\Dom)}^\half
    \left(
        \rho_*^{-{M_1}} + 2^{-\Linv}
    \right)
    \lesssim \rho_*^{-{M_1}}.
    \end{aligned}
\eeq
Thus we can choose $M_1 \sim | \log \eps |$ for the error to be within the threshold $\eps$. Observe that the total number of degree of freedom in our approximation is $M \lesssim M_1+M_2 + \Linv \lesssim M_1$, so the claim is proved.
\end{proof}

\subsection{The singular and nonlinear case: the Burgers' equation}

In this section, we will show that the strategy of separately approximating the initial condition and the smooth characteristic curves, used in \cref{thm:construct} cannot apply to nonlinear problems that possess characteristic curves that are singular. However, one may still find an RDN that approximates the solution manifold with small degrees of freedom, simply by utilizing more hidden layers.

We will simplify our discussion by considering a single representative initial value problem with a particular monotonically non-increasing initial condition. The results in this subsection regarding the the Kolmogorov $N$-widths and the RDN construction for the Burgers' equation are not restricted to this simple setting, and apply to solution manifold with initial conditions in $\UU_0 \subset \UU$ satisfying \cref{as:CC} (i-iii), upon suitable localizations of the solution manifold.

\subsubsection{Shock formation and singular characteristics}

We consider the solution manifold of the Burgers' equation. This section relies on well-established facts about the equation, such as weak solutions, Rankine-Hugoniot jump conditions, shock formation, found in standard references \cite{lax72,fvmbook,evans10}. 

Consider the PDE \cref{eq:pde} with the flux function $f(u,x;\bmu) = \half u^2$ and no source term $\psi(u,x;\bmu) = 0$. We will fix the initial condition $u_0 \in \mathbb{U}$ given by
\beq
    u_0(x)
    =
    \left\{
    \begin{aligned}
        0,  \quad & x > x_0 + \gamma,
        \\
        \half - \half \sin(\frac{\pi}{2\gamma}(x - x_0)),
            \quad & \abs{x - x_0} \le \gamma,
        \\
        1,  \quad & x < x_0 - \gamma,
    \end{aligned}
    \right.
\eeq
in which we set $\gamma = 0.2$. We choose the final time $\tfin = 3$ and denote the solution manifold by
\beq
    \cM_B := 
    \left\{
        u(x,t): t \in [0,\tfin]
    \right\}.
\eeq

Let $X(t,x)$ denote the characteristic curves for the solution in $\cM_B$. The time shock appears is the smallest time $t_1 \in [0,\tfin]$ when $\partial_x X > 0$ fails to hold: this is when the following equation has multiple roots in $t$,
\beq
    \frac{1}{t} + u_0'(x) 
    = 
    \frac{1}{t} - \frac{\pi}{4\gamma} \cos( \frac{\pi}{2\gamma}(x - x_0)) 
    =
    0
\eeq
so $t_1 = 4\gamma / \pi$. At a later time, a shock of unit height forms, then the solution becomes a single jump that propagates at constant speed. We will denote by $t_2 \in (t_1,\tfin)$ the time at which the shock formation is complete. 

Denote by $x_S(t)$ the shock location derived from the Rankine-Hugoniot jump condition. Let us denote,
\beq
    \Gamma(t)
    :=
    \left\{
        x \in \RR
        \,|\,
        x_S(t) = \check{X}(t;x)
    \right\},
    \quad
    I(t) := (\inf \Gamma(t), \sup \Gamma(t))
\eeq
where we define $I(t)$ to be empty before shock formation $(t < t_1)$. Using this notation, the characteristics for the weak solution is given by 
\beq
    X(t,x)
    :=
    \begin{cases}
        x + u_0(x) t, \quad & x \notin    I(t), \\
        x_S(t),               & x \in I(t).
    \end{cases}
\eeq
So $X(t, \cdot)$ is continuous, piecewise analytic and is continuous in $t$, and satisfies $\partial_x X \ge 0$. 

To show that $\cM_B$ is $\half$-sharply convective, one simply observes that the solution is a traveling jump function when $t > t_2$, which reduces the problem to the case of $\cM_A$ (\cref{lem:OR}) up to scaling. Now, we show that the collection of these characteristic curves $X(t,\cdot)$ themselves form a sharply convective class, in contrast to the regular and linear case in  \cref{sec:color}.

\begin{figure}
    \centering
    \includegraphics[width=0.8\textwidth]{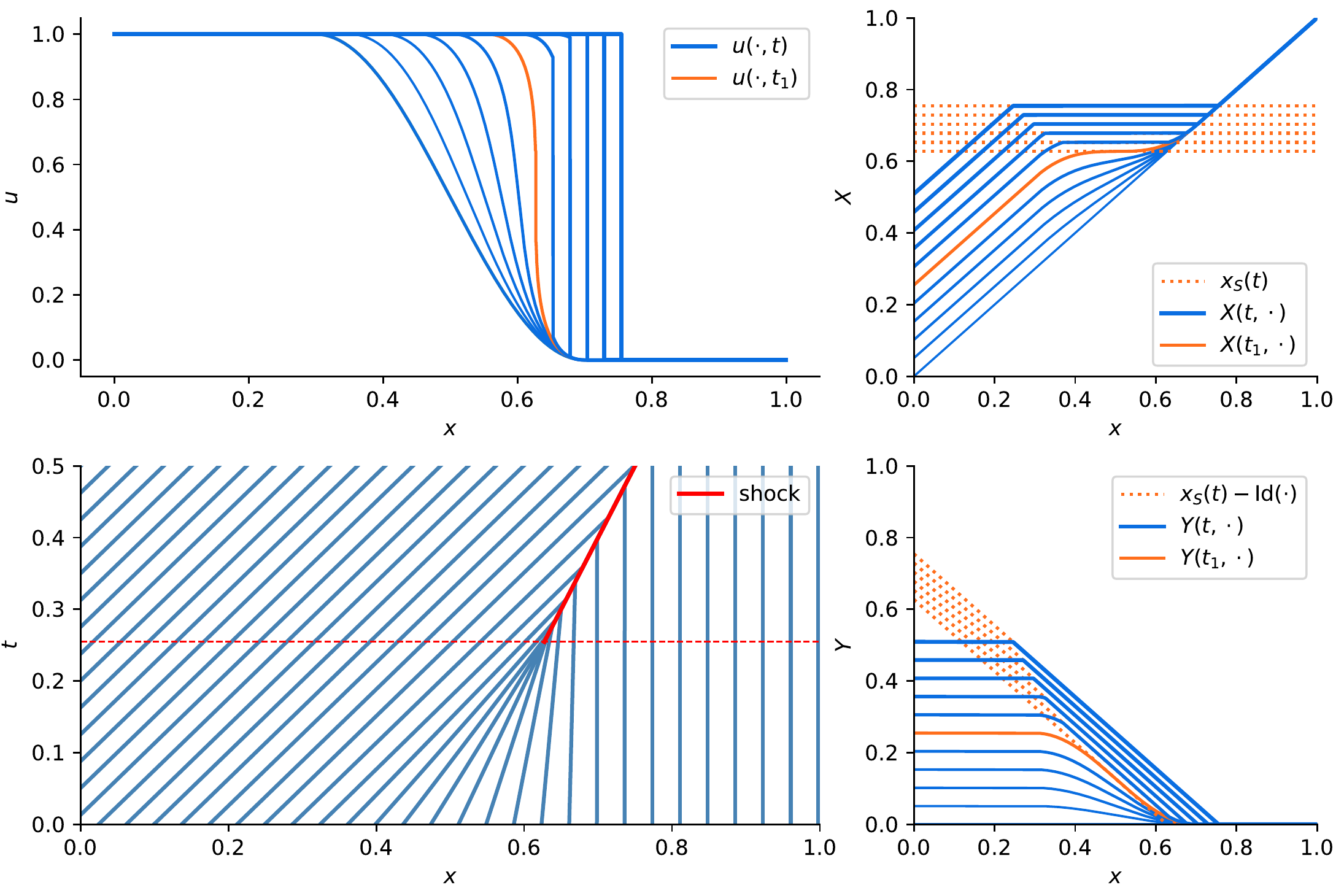}
    \caption{Burgers solution (top left) and the evolution of its transport map $X(t,\cdot)$ (top right), characteristic curves (bottom left), and the maps $Y(t,\cdot) := X(t,\cdot) - \Id(\cdot)$ (bottom right).}
    \label{fig:burgers}
\end{figure}

\begin{theorem}\label{lem:CB1}
$\cM_B^{X} := \{X(t,x): t \in (0,\tfin)\}$ has $d(N;\cM_B^{X}) \gtrsim N^{-\frac{3}{2}}$.
\end{theorem}

\begin{proof}
Denote $\Dom_S := \{x \in \Dom: x > \inf I(t_S)\}$, and let $Y(t,\cdot) := X(t,\cdot) - X(0,\cdot)$. We will consider the two time intervals $(t_1,t_2)$ and $(t_2,t_F)$ separately. While it is sufficent to prove the lower bound for either one of these intervals, we will provide a proof for both cases.

(\emph{The case $t \in (t_1,t_2)$}) $Y(t,\cdot)|_{\Dom_S}$ is piecewise analytic with separate pieces in $\Dom_S \cap I(t)$ and $\Dom_S \cap [\sup I(t), 1)$. It is $Y(t,\cdot) = x_S(t) - x$ in the former where it is linear, whereas $Y(t,\cdot) = u_0(\cdot)t$ in the latter where it is analytic. At the point $\sup I(t)$, $\p_x Y$ has a jump. Hence, we apply the arguments of the proof of \cref{thm:colorlb} with minor changes. Applying the finite difference stencil $( b_{k} )_{k=1}^K$ over the equidistant grid $\tau_{nk}$ near $\tau_n$ \cref{eq:taunk} for sufficiently large $K$ and $s_1 = 2$ in \cref{eq:vandermonde}, we take $\phi_n(x) := \sum_{k=1}^K b_k Y(\tau_{nk},x)$ for $n = 1, ... , 2N$. Then  $\abs{\phi_n} \sim N^{-1}$ in the neighborhood $S_n$ of $X(t_n,x_0)$ of measure $\sim N^{-1}$, with $\{S_n\}_{n=1}^{2N}$ mutually disjoint, and one has $(\phi_n,\phi_m)_{\Dom_1 \setminus (S_n \cup S_m)} \lesssim (\Delta t)^{2 s_1}$.  Using the first part of \cref{lem:sharp} and a Gram-Schmidt process, we obtain an orthogonal $\{\psi_n\}_{n=1}^{2N}$ with $\Norm{\psi_n}{L^2(\Dom_1)} \gtrsim N^{-\frac{3}{2}}$. Therefore, we obtain the result by the second part of \cref{lem:sharp}. 

(\emph{Case $t \in (t_2,t_F)$}) We have that $Y(t,\cdot)$ is linear and supported in $I(t)$, is zero for $x > I(t)$, from which a direct argument follows. We choose $\{\tau_{nk}\}_{k=1}^K$, with which we can construct for $n = 1, ... , 2N$, $\phi_n(x) := \sum_{k=1}^K a_{nk} Y(\tau_{nk},x)$ so that the set $S_n := \supp \phi_n \cap \Dom_S$ with $|S_n| \sim N^{-1}$ is pairwise disjoint in $n$. In particular, we may take $(a_{n1}, a_{n2}, a_{n3}) = (-1,2,-1)$ with $(\tau_{n1},\tau_{n2},\tau_{n3})$ chosen so that
\beq
    x_S(\tau_{nk}) = x_S(t_S) + (3(n-1) + (k-1))\nu, 
    \quad
    \nu := \frac{|\Dom_S|}{6N}
    \quad
    k = 1,2,3.
\eeq
Then $\Norm{\phi_n}{L^2(\Dom_1)} \gtrsim  N^{-\frac{3}{2}}$ and $\phi_n$ has support in $S_n = (x_S(\tau_{n1}),x_S(\tau_{n3}))$, and therefore pairwise disjoint. Therefore $\cM_B^X$ is $\frac{3}{2}$-sharply convective, and the result follows by \cref{lem:sharp}.
\end{proof}

\subsubsection{A deep RDN approximation of singular characteristics}

Now, we construct an efficient RDN approximation of the solution manifold $\cM_B$ exhibiting singular characteristics. The construction is obtained by using more layers.

\begin{theorem}
For any given error threshold $\eps \in (0,1)$, there exists an reduced deep network solution manifold $\cMrd \subset \cNrmats$ with total degrees of freedom  $M \lesssim \abs{\log \eps}$.
\end{theorem}

\begin{proof}
Let us define,
\beq
    \begin{aligned}
        T_{11}(x;t) &:= 
            x + w_{111}(t) \varsigma(x + w_{112}(t)),
        &
        T_1(x;t) &:= T_{12} \circ T_{11} (x;t),\\
        T_{12}(x;t) &:= x + w_{12}(t) u_0(x),
        &
        T_2(x;t) &:= x + w_{21}(t) \varsigma(x + w_{22}(t)).
    \end{aligned}
    \label{eq:RDN-burgers}
\eeq
For $t \in [0,\tfin]$, we choose the weights
\beq
    w_{12}(t) = t,
    \quad
    w_{111}(t) = w_{21}(t) = |I(t)|,
    \quad
    w_{112}(t) = w_{22}(t) = - x_S(t).
    \label{eq:wgts-burgers}
\eeq
Note that $T_{11}', T_2' \ge 0$. Let us construct,
\beq
    \Tlzs(x;t) := x + w_{12}(t) \uos(x),
    \quad
    \Tls(x;t) := \Tlzs \circ T_{11} (x;t),
    \label{eq:RDN-burgers1}
\eeq
in which $\Tlzs$ is an approximation to $T_{12}$, obtained by approximating $u_0 \in \UU$ by $\uos \in \cNs \cap \cNd_\sharp$, satisfying
\beq
    \Normlr{ T_1^{-1} - \Tls^{-1}}{L^\infty(\Dom)} 
    =
    \Normlr{ (T_{12} \circ T_{11} )^{-1} -( \Tlzs \circ T_{11} )^{-1}}{L^\infty(\Dom)} 
    < 
    \eps_1.
\eeq

Observe the true solution is $u = u_0 \circ T_2 \circ T_1^{-1}$ with the weights \cref{eq:wgts-burgers}. We will now show that $\urd \in \cNmats$ of the form
\beq
    \urd 
    := 
    \uos \circ T_2 \circ \Tls^\flat 
    = 
    \uos \circ T_2 \circ (\Tlzs \circ T_{11})^\flat 
\eeq
satisfies
\beq
    \Normlr{u_0 \circ T_2 \circ T_1^{-1}
          -
          \uos \circ T_2 \circ \Tls^\flat
         }{\VV}
    < \eps.
\eeq
Next, we have
\beq
    \begin{aligned}
    &\Normlr{u_0 \circ T_2 \circ T_1^{-1}
           -
           \uos \circ T_2 \circ \Tls^{\flat}
          }{\VV}
    \\&\quad \le 
    \Normlr{
          u_0 \circ T_2 \circ T_1^{-1}
          -
          u_0 \circ T_2 \circ \Tls^{\flat}
         }{\VV}
    + 
    \Normlr{
          u_0 \circ T_2 \circ \Tls^{\flat}
          -
          \uos \circ T_2 \circ \Tls^{\flat}
         }{\VV}
    \\&\quad \le 
    \Normlr{u'_0}{L^\infty(\Dom)}
    \Normlr{ T_2 \circ T_1^{-1} - T_2 \circ \Tls^{\flat}}{\VV}
    \\ & \qquad + 
    \left( \Norm{u'_0}{L^\infty(\Dom)} + \Norm{\uos'}{L^\infty(\Dom)} \right)
    \abs{\Dom} 2^{-\Linv}
    +
    \Norm{ (T_2 \circ \Tls^{-1})' }{L^\infty(\Dom)}^\half
    \Normlr{ u - \uos}{\VV}  .
    \end{aligned}
\eeq
In the first term, due to $T_2$ being an identity almost everywhere,
\beq
    \Normlr{ T_2 \circ T_1^{-1} - T_2 \circ \Tls^{\flat}}{\VV}
    \le
    \left(\abs{\Dom} + \abs{I(t)} \right) 
    \Normlr{T_1^{-1} - \Tls^{\flat}}{L^\infty(\Dom)},
\eeq
in which we have,
\beq
    \Normlr{ T_1^{-1} - \Tls^{-1}  }{L^\infty(\Dom)}
    +
    \Normlr{ T_1^{-1} - \Tls^\flat }{L^\infty(\Dom)}
    <
    \eps_1
    +
    |\Dom| 2^{-\Linv}.
\eeq
In the second term, we may suppose $\Normlr{u_0 - \bar{u}_0}{\VV} < \eps_1$, so we have for $\Linv = \log \eps_1$ and $\eps_1 \sim \eps$ small enough
\beq
    \Normlr{u_0 \circ T_2 \circ T_1^{-1}
           -
           \uos \circ T_2 \circ \Tls^{\flat}
          }{\VV} < \eps.
\eeq
Let $M$ denote the total degrees of freedom. Then counting the number of weights in \cref{eq:RDN-burgers} as well as the number of inversion layers, we have $M \lesssim \Linv \lesssim |\log \eps|$. 
\end{proof}

\section*{Acknowledgements} The work of the first author (Rim) and fourth author (Peherstorfer) was partially supported by the Air Force Center of Excellence on Multi-Fidelity Modeling of Rocket Combustor Dynamics under Award Number FA9550-17-1-0195 and AFOSR MURI on multi-information sources of multi-physics systems under Award Number FA9550-15-1-0038 (Program Manager Dr.~Fariba Fahroo). The fourth author (Peherstorfer) was additionally partially supported by the National Science Foundation under Grant No.~1901091. The work of the second author (Venturi) and the third author (Bruna) was partially supported by the \mbox{Alfred} P. Sloan Foundation, NSF RI-1816753, NSF CAREER CIF 1845360, NSF \mbox{CHS-1901091}, Samsung Electronics, and the Institute for Advanced Study. 

The first author (Rim) thanks Gerrit Welper and Weilin Li for fruitful discussions.

\appendix

\section{Proof of \cref{lem:inv}}\label{sec:proof:lem:inv}

We provide an explicit construction that implements the bisection method. Given the neural network $f$, we first construct a neural network $g_f$ with input and output 
\[
    g_f([a,b,x]) = [a',b',x],
\]
supposing that $\Dom = (a,b)$ and $x \in f(\Dom)$. First $L+1$ layers are given by
\[
    \begin{bmatrix}
    \Id \\
    \Id \\
    \Id \\
    f 
    \end{bmatrix}
    \odot
    \begin{bmatrix} 
        1 & 0 & 0 \\
        0 & 1 & 0 \\
        0 & 0 & 1 \\ 
        \half & \half & 0 \\ 
    \end{bmatrix}
    \begin{bmatrix} a \\ b \\ x \end{bmatrix}
    =
    \begin{bmatrix} a \\ b \\ x \\ f(\frac{a+b}{2}) \end{bmatrix},
\]
in which $f$ appears as an activation for ease of exposition, although it actually is itself a network with $L$ layers.  Layer $L+2$ is given by
\[
    \begin{bmatrix}
    \Id \\ \Id \\ \Id \\ \varsigma
    \end{bmatrix}
    \odot
    \begin{bmatrix}
    1 & 0 & 0 & 0 \\
    0 & 1 & 0 & 0 \\
    0 & 0 & 1 & 0 \\
    0 & 0 &-1 & 1 
    \end{bmatrix}
    \begin{bmatrix}
    a \\ b \\ x \\ \tau
    \end{bmatrix}
    =
    \begin{bmatrix}
    a \\ b \\ x \\ \varsigma(\tau - x)
    \end{bmatrix},
\]
layer $L+3$ is given by ($c := b-a = |\Dom|$)
\[
    \begin{bmatrix}
    \Id \\ \Id \\ \Id \\ \sigma \\ \sigma
    \end{bmatrix}
    \odot
    \begin{bmatrix}
    1 & 0 & 0 & 0 \\
    0 & 1 & 0 & 0 \\
    0 & 0 & 1 & 0 \\
    -\half & \half & 0 & -c \\
    -\half & \half & 0 & c 
    \end{bmatrix}
    \begin{bmatrix}
    a \\ b \\ x \\ w
    \end{bmatrix}
    +
    \begin{bmatrix}
    0 \\ 0 \\ 0 \\ 0 \\ -c 
    \end{bmatrix}
    =
    \begin{bmatrix}
    a \\ b \\ x \\ 
    \sigma(-c \gamma    + \frac{b-a}{2})\\ 
    \sigma( c(\gamma-1) + \frac{b-a}{2})
    \end{bmatrix},
\]
and layer $L+4$
\[
    \begin{bmatrix}
    \Id \\ \Id \\ \Id \\ \Id \\ \Id
    \end{bmatrix}
    \odot
    \begin{bmatrix}
    1 & 0 & 0 & 1 &  0 \\
    1 & 0 & 0 & 0 & -1 \\
    0 & 0 & 1 & 0 &  0 \\
    \end{bmatrix}
    \begin{bmatrix}
    a \\ b \\ x \\ s_1 \\ s_2
    \end{bmatrix}
    =
    \begin{bmatrix}
    a + s_1 \\ a - s_2 \\ x
    \end{bmatrix}.
\]
Thusly defined $g_f$ has $L+4$ layers. Then $\Linv$ compositions of $g_f$ 
\[
    \underbrace{g_f \circ g_f \circ ... \circ g_f}_{\Linv \text{ times}}
    =
    [a,b,x]^T
\]
outputs the values $a,b$ whose distance to $f^{-1}(x)$ is less than $|\Dom| 2^{-\Linv}$. Taking the mid-point,
\[
    \begin{aligned}
    \mathring{g}_f(x)
    &:=
    \left[ \half, \half, 0 \right]^T (g \circ g \circ ... \circ g)
    ([1,0,0]^T(x) + [0,a,b])
    = 
    \half (a + b)
    = 
    y_*,
    \end{aligned}
\]
for which it holds $|f^{-1}(x) - y_*| \le |\Dom| 2^{-L}$.  Let $f^{\flat}(x) := \mathring{g}_f(x)$.

\section{Proof of \cref{lem:analytic}}\label{sec:proof:lem:analytic}

Let 
\[
    x_1 := X(t_1,x_0) = x_0 + \int_0^{t_1} c(X(\tau, x_0)) \dtau.
\]
Then the partial derivative of $X$ with respect to $x$ is,
\[
    \begin{aligned}
    \lim_{x_1 \to x_0} &
    \frac{ X(t,x_1) - X(t,x_0)}
         { x_1 - x_0}
    \\&=
    \lim_{x_1 \to x_0}
    \frac{ (x_1 + \int_0^t c(X(\tau, x_1)) \dtau)
         - (x_0 + \int_0^t c(X(\tau, x_0)) \dtau)}
         {x_1 - x_0}
    \\&=
    \lim_{x_1 \to x_0}
    \frac{    \int_0^{t_1} c(X(\tau,x_0)) \dtau 
            + \int_0^t c(X(\tau,x_1)) \dtau 
            - \int_0^t c(X(\tau,x_0)) \dtau 
         }
         {x_1 - x_0}
    \\&=
    \lim_{x_1 \to x_0}
    \frac{  \int_0^{t+t_1} c(X(\tau,x_1)) \dtau 
          - \int_0^t       c(X(\tau,x_0)) \dtau 
         }
         {x_1 - x_0}
    \\&=
    \left(
    \lim_{t_1 \to t_0}
    \frac{  \int_t^{t_1} c(X(\tau,x_0)) \dtau }
         {t_1 - t_0}
    \right)
    \cdot
    \left(
    \lim_{x_1 \to x_0}
    \frac{t_1 - t_0}
         {x_1 - x_0}
    \right)
    =
    \frac{c(X(t,x_0))}{c(x_0)}.
    \end{aligned}
\]
By the semi-group property of the solution $X(t_1 + t, x_0) = X(t, x_1)$, so
that
\[
    \int_0^t c(X(\tau,x_1)) \dtau = \int_{t_1}^{t+ t_1} c(X(\tau, x_0)) \dtau.
\]
So $X(t,\cdot)$ is the solution to the ODE
\[
    \left\{
    \begin{aligned}
    \partial_x X(t,x) &= \tilde{c}(X,x) := \frac{c(X(t,x))}{c(x)},
    \\
    X(t,x_0) &= x_0.
    \end{aligned}
    \right.
\]
Since $c > c_0 > 0$, $\tilde{c}$ is also analytic.

\bibliographystyle{siamplain}

\begin{thebibliography}{10}

\bibitem{amsallem16}
{\sc R.~Abgrall, D.~Amsallem, and R.~Crisovan}, {\em Robust model reduction by
  {$L^{1}$}-norm minimization and approximation via dictionaries: application
  to nonlinear hyperbolic problems}, Advanced Modeling and Simulation in
  Engineering Sciences, 3 (2016).

\bibitem{Arora18}
{\sc S.~Arora, R.~Ge, B.~Neyshabur, and Y.~Zhang}, {\em Stronger generalization
  bounds for deep nets via a compression approach}, in International Conference
  on Machine Learning, 2018.

\bibitem{black19}
{\sc F.~Black, P.~Schulze, and B.~Unger}, {\em Nonlinear {G}alerkin model
  reduction for systems with multiple transport velocities}, Preprint,  (2019),
  \href{http://arxiv.org/abs/1912.11138}{arXiv:1912.11138}.

\bibitem{Bolcskei19}
{\sc H.~B\"{o}lcskei, P.~Grohs, G.~Kutyniok, and P.~Petersen}, {\em Optimal
  approximation with sparsely connected deep neural networks}, SIAM Journal on
  Mathematics of Data Science, 1 (2019), pp.~8--45.

\bibitem{Bucilua06}
{\sc C.~Bucilu\v{a}, R.~Caruana, and A.~Niculescu-Mizil}, {\em Model
  compression}, in Proceedings of the 12th ACM SIGKDD International Conference
  on Knowledge Discovery and Data Mining, Association for Computing Machinery,
  2006.

\bibitem{Cagniart2019}
{\sc N.~Cagniart, Y.~Maday, and B.~Stamm}, {\em Model Order Reduction for
  Problems with Large Convection Effects}, Springer International Publishing,
  2019, pp.~131--150.

\bibitem{carlberg15}
{\sc K.~Carlberg}, {\em Adaptive $h$-refinement for reduced-order models},
  International Journal for Numerical Methods in Engineering, 102 (2015),
  pp.~1192--1210.

\bibitem{chen15}
{\sc W.~Chen, J.~Wilson, S.~Tyree, K.~Weinberger, and Y.~Chen}, {\em
  Compressing neural networks with the hashing trick}, in Proceedings of the
  32nd International Conference on Machine Learning, vol.~37 of Proceedings of
  Machine Learning Research, PMLR, 2015, pp.~2285--2294.

\bibitem{Cheng18}
{\sc Y.~{Cheng}, D.~{Wang}, P.~{Zhou}, and T.~{Zhang}}, {\em Model compression
  and acceleration for deep neural networks: The principles, progress, and
  challenges}, IEEE Signal Processing Magazine, 35 (2018), pp.~126--136.

\bibitem{coddlev55}
{\sc E.~A. Coddington and N.~Levinson}, {\em Theory of Ordinary Differential
  Equations}, McGraw-Hill, New York, NY, USA, 1955.

\bibitem{Cybenko89}
{\sc G.~Cybenko}, {\em Approximation by superpositions of a sigmoidal
  function}, Mathematics of Control, Signals and Systems, 2 (1989), pp.~303 --
  314.

\bibitem{Daubechies19}
{\sc I.~Daubechies, R.~DeVore, S.~Foucart, B.~Hanin, and G.~Petrova}, {\em
  Nonlinear approximation and (deep) {R}e{L}{U} networks}, 2019,
  \href{http://arxiv.org/abs/1905.02199}{arXiv:1905.02199}.

\bibitem{donoho01}
{\sc D.~L. Donoho}, {\em Sparse components of images and optimal atomic
  decompositions}, Constr. Approx., 17 (2001), pp.~353--382,
  \href{http://dx.doi.org/10.1007/s003650010032}{doi:\nolinkurl{10.1007/s003650010032}}.

\bibitem{ehrlacher19}
{\sc V.~Ehrlacher, D.~Lombardi, O.~Mula, and F.-X. Vialard}, {\em Nonlinear
  model reduction on metric spaces. {A}pplication to one-dimensional
  conservative {P}{D}{E}s in {W}asserstein spaces}, Preprint,  (2019),
  \href{http://arxiv.org/abs/1909.06626}{arXiv:1909.06626}.

\bibitem{Eldan16}
{\sc R.~Eldan and O.~Shamir}, {\em The power of depth for feedforward neural
  networks}, in 29th Annual Conference on Learning Theory, V.~Feldman,
  A.~Rakhlin, and O.~Shamir, eds., vol.~49 of Proceedings of Machine Learning
  Research, 23--26 Jun 2016, pp.~907--940.

\bibitem{evans10}
{\sc L.~C. Evans}, {\em Partial differential equations}, American Mathematical
  Society, Providence, R.I., 2nd~ed., 2010.

\bibitem{Geist20}
{\sc M.~Geist, P.~Petersen, M.~Raslan, R.~Schneider, and G.~Kutyniok}, {\em
  Numerical solution of the parametric diffusion equation by deep neural
  networks}, Preprint,  (2020),
  \href{http://arxiv.org/abs/2004.12131}{arXiv:2004.12131}.

\bibitem{Gerbeau14}
{\sc J.-F. Gerbeau and D.~Lombardi}, {\em Approximated lax pairs for the
  reduced order integration of nonlinear evolution equations}, Journal of
  Computational Physics, 265 (2014), pp.~246 -- 269.

\bibitem{golub96}
{\sc G.~H. Golub and C.~F. Van~Loan}, {\em Matrix Computations (3rd Ed.)},
  Johns Hopkins University Press, Baltimore, MD, USA, 1996.

\bibitem{Greif19}
{\sc C.~Greif and K.~Urban}, {\em Decay of the {K}olmogorov {$N$}-width for
  wave problems}, Applied Mathematics Letters, 96 (2019), pp.~216 -- 222.

\bibitem{crb-book}
{\sc J.~S. Hesthaven, G.~Rozza, and B.~Stamm}, {\em Certified Reduced Basis
  Methods for Parametrized Partial Differential Equations}, Springer Cham,
  Cham, Switzerland, 2016.

\bibitem{Hinton15D}
{\sc G.~E. Hinton, O.~Vinyals, and J.~Dean}, {\em Distilling the knowledge in a
  neural network}, ArXiv, abs/1503.02531 (2015).

\bibitem{horn12}
{\sc R.~Horn and C.~Johnson}, {\em Matrix Analysis}, Cambridge University
  Press, 2nd~ed., 2012.

\bibitem{iollo14}
{\sc A.~Iollo and D.~Lombardi}, {\em Advection modes by optimal mass transfer},
  Phys. Rev. E, 89 (2014), p.~022923.

\bibitem{koch07}
{\sc O.~Koch and C.~Lubich}, {\em Dynamical low‐rank approximation}, SIAM
  Journal on Matrix Analysis and Applications, 29 (2007), pp.~434--454.

\bibitem{Kutyniok19}
{\sc G.~Kutyniok, M.~R. Philipp~Petersen, and R.~Schneider}, {\em A theoretical
  analysis of deep neural networks and parametric {P}{D}{E}s}, Preprint,
  (2019), \href{http://arxiv.org/abs/1904.00377}{arXiv:1904.00377}.

\bibitem{Laakmann20}
{\sc F.~Laakmann and P.~Petersen}, {\em Efficient approximation of solutions of
  parametric linear transport equations by {R}e{L}{U} {D}{N}{N}s},  (2020),
  \href{http://arxiv.org/abs/2020.11441}{arXiv:2020.11441}.

\bibitem{lax72}
{\sc P.~D. Lax}, {\em Hyperbolic Systems of Conservation Laws and the
  Mathematical Theory of Shock Waves}, vol.~11, Society for Industrial and
  Applied Mathematics, 1972.

\bibitem{LeCun89}
{\sc Y.~LeCun, J.~S. Denker, and S.~A. Solla}, {\em Optimal brain damage}, in
  Advances in Neural Information Processing Systems 2, 1990, pp.~598--605.

\bibitem{Lee20}
{\sc K.~Lee and K.~T. Carlberg}, {\em Model reduction of dynamical systems on
  nonlinear manifolds using deep convolutional autoencoders}, Journal of
  Computational Physics, 404 (2020), p.~108973.

\bibitem{fvmbook}
{\sc R.~J. LeVeque}, {\em Finite Volume Methods for Hyperbolic Problems},
  Cambridge University Press, Cambridge, 1st~ed., 2002.

\bibitem{fdmbook}
{\sc R.~J. LeVeque}, {\em Finite Difference Methods for Ordinary and Partial
  Differential Equations: Steady-State and Time-Dependent Problems (Classics in
  Applied Mathematics)}, Society for Industrial and Applied Mathematics, USA,
  2007.

\bibitem{mendible19}
{\sc A.~Mendible, S.~L. Brunton, A.~Y. Aravkin, W.~Lowrie, and J.~N. Kutz},
  {\em Dimensionality reduction and reduced order modeling for traveling wave
  physics}, Preprint,  (2019),
  \href{http://arxiv.org/abs/1911.00565}{arXiv:1911.00565}.

\bibitem{Mojgani17}
{\sc R.~Mojgani and M.~Balajewicz}, {\em Lagrangian basis method for
  dimensionality reduction of convection dominated nonlinear flows}, Preprint,
  \href{http://arxiv.org/abs/1701.04343}{arXiv:1701.04343}.

\bibitem{musharbash20}
{\sc E.~Musharbash, F.~Nobile, and E.~Vidličková}, {\em Symplectic dynamical
  low rank approximation of wave equations with random parameters}, BIT
  Numerical Mathematics,  (2020).

\bibitem{Neyshabur17}
{\sc B.~Neyshabur, S.~Bhojanapalli, D.~A. McAllester, and N.~Srebro}, {\em A
  {P}{A}{C}-bayesian approach to spectrally-normalized margin bounds for neural
  networks}, Preprint,  (2017),
  \href{http://arxiv.org/abs/1707.09564}{arXiv:1707.09564}.

\bibitem{nonino19}
{\sc M.~Nonino, F.~Ballarin, G.~Rozza, and Y.~Maday}, {\em Overcoming slowly
  decaying {K}olmogorov n-width by transport maps: application to model order
  reduction of fluid dynamics and fluid--structure interaction problems},
  Preprint,  (2019), \href{http://arxiv.org/abs/1911.06598}{arXiv:1911.06598}.

\bibitem{novikov15}
{\sc A.~Novikov, D.~Podoprikhin, A.~Osokin, and D.~P. Vetrov}, {\em Tensorizing
  neural networks}, in Advances in Neural Information Processing Systems 28,
  2015, pp.~442--450.

\bibitem{Ohlberger13}
{\sc M.~Ohlberger and S.~Rave}, {\em Nonlinear reduced basis approximation of
  parameterized evolution equations via the method of freezing}, Comptes Rendus
  Mathematique, 351 (2013), pp.~901 -- 906.

\bibitem{Ohlberger16}
{\sc M.~Ohlberger and S.~Rave}, {\em Reduced basis methods: Success,
  limitations and future challenges}, Proceedings of the Conference Algoritmy,
  (2016), pp.~1--12.

\bibitem{siamrev-survey}
{\sc S.~G. P.~Benner and K.~Willcox}, {\em A survey of projection-based model
  reduction methods for parametric dynamical systems}, SIAM Rev., 57 (2015),
  pp.~483--531.

\bibitem{P18AADEIM}
{\sc B.~Peherstorfer}, {\em Model reduction for transport-dominated problems
  via online adaptive bases and adaptive sampling}, SIAM Journal on Scientific
  Computing,  (2020).

\bibitem{pehersto15}
{\sc B.~Peherstorfer and K.~Willcox}, {\em Online adaptive model reduction for
  nonlinear systems via low-rank updates}, SIAM Journal on Scientific
  Computing, 37 (2015), pp.~A2123--A2150,
  \href{http://dx.doi.org/10.1137/140989169}{doi:\nolinkurl{10.1137/140989169}}.

\bibitem{pinkus12}
{\sc A.~Pinkus}, {\em $n$-{W}idths in {A}pproximation {T}heory}, Springer,
  Berlin, Heidelberg, 1985.

\clearpage

\bibitem{Raissi19}
{\sc M.~Raissi, P.~Perdikaris, and G.~Karniadakis}, {\em Physics-informed
  neural networks: A deep learning framework for solving forward and inverse
  problems involving nonlinear partial differential equations}, Journal of
  Computational Physics, 378 (2019), pp.~686 -- 707.

\bibitem{Regazzoni19}
{\sc F.~Regazzoni, L.~Dedè, and A.~Quarteroni}, {\em Machine learning for fast
  and reliable solution of time-dependent differential equations}, Journal of
  Computational Physics, 397 (2019), p.~108852.

\bibitem{schulze18}
{\sc J.~Reiss, P.~Schulze, J.~Sesterhenn, and V.~Mehrmann}, {\em The shifted
  proper orthogonal decomposition: A mode decomposition for multiple transport
  phenomena}, SIAM Journal on Scientific Computing, 40 (2018),
  pp.~A1322--A1344.

\bibitem{radonsplit}
{\sc D.~Rim}, {\em Dimensional splitting of hyperbolic partial differential
  equations using the {R}adon transform}, SIAM Journal on Scientific Computing,
  40 (2018), pp.~A4184--A4207,
  \href{http://dx.doi.org/10.1137/17M1135633}{doi:\nolinkurl{10.1137/17M1135633}}.

\bibitem{rim18mr}
{\sc D.~Rim and K.~Mandli}, {\em Displacement interpolation using monotone
  rearrangement}, SIAM/ASA Journal on Uncertainty Quantification, 6 (2018),
  pp.~1503--1531.

\bibitem{rim17reversal}
{\sc D.~Rim, S.~Moe, and R.~LeVeque}, {\em Transport reversal for model
  reduction of hyperbolic partial differential equations}, SIAM/ASA Journal on
  Uncertainty Quantification, 6 (2018), pp.~118--150.

\bibitem{Rim19}
{\sc D.~Rim, B.~Peherstorfer, and K.~T. Mandli}, {\em {M}anifold
  {A}pproximations via {T}ransported {S}ubspaces: Model reduction for
  transport-dominated problems}, Preprint, {\tt arXiv:1912.13024. [math.NA]}
  (2019), \href{http://arxiv.org/abs/1912.13024}{arXiv:1912.13024}.

\bibitem{rowley00}
{\sc C.~W. Rowley and J.~E. Marsden}, {\em Reconstruction equations and the
  {Karhunen-Lo\`eve} expansion for systems with symmetry}, Physica D,  (2000),
  pp.~1--19.

\bibitem{sapsis09}
{\sc T.~P. Sapsis and P.~F. Lermusiaux}, {\em Dynamically orthogonal field
  equations for continuous stochastic dynamical systems}, Physica D: Nonlinear
  Phenomena, 238 (2009), pp.~2347 -- 2360.

\bibitem{Schwab19}
{\sc C.~Schwab and J.~Zech}, {\em Deep learning in high dimension: Neural
  network expression rates for generalized polynomial chaos expansions in
  {U}{Q}}, Analysis and Applications, 17 (2019), pp.~19--55.

\bibitem{sesterhenn2019}
{\sc J.~Sesterhenn and A.~Shahirpour}, {\em A characteristic dynamic mode
  decomposition}, Theoretical and Computational Fluid Dynamics, 33 (2019),
  pp.~281--305.

\bibitem{Strang73}
{\sc G.~Strang and G.~J. Fix}, {\em {An analysis of the finite element
  method}}, Prentice-Hall series in automatic computation, Prentice-Hall,
  Englewood Cliffs, NJ, 1973, \url{https://cds.cern.ch/record/102774}.

\bibitem{taddei19}
{\sc T.~Taddei}, {\em A registration method for model order reduction: Data
  compression and geometry reduction}, SIAM Journal on Scientific Computing, 42
  (2020), pp.~A997--A1027.

\bibitem{taddei14}
{\sc T.~Taddei, S.~Perotto, and A.~Quarteroni}, {\em Reduced basis techniques
  for nonlinear conservation laws}, ESAIM: M2AN, 49 (2015), pp.~787--814.

\bibitem{Telgarsky16}
{\sc M.~Telgarsky}, {\em Benefits of depth in neural networks}, in 29th Annual
  Conference on Learning Theory, V.~Feldman, A.~Rakhlin, and O.~Shamir, eds.,
  vol.~49 of Proceedings of Machine Learning Research, Columbia University, New
  York, New York, USA, 2016, PMLR, pp.~1517--1539.

\bibitem{temlyakov08}
{\sc V.~N. Temlyakov}, {\em Greedy approximation}, Acta Numerica, 17 (2008),
  p.~235–409.

\bibitem{trefethen}
{\sc L.~N. Trefethen}, {\em Approximation Theory and Approximation Practice},
  Society for Industrial and Applied Mathematics, Philadelphia, PA, USA, 2012.

\bibitem{Wang19}
{\sc Q.~Wang, J.~S. Hesthaven, and D.~Ray}, {\em Non-intrusive reduced order
  modeling of unsteady flows using artificial neural networks with application
  to a combustion problem}, Journal of Computational Physics, 384 (2019),
  pp.~289 -- 307.

\bibitem{welper17}
{\sc G.~Welper}, {\em Interpolation of functions with parameter dependent jumps
  by transformed snapshots}, SIAM Journal on Scientific Computing, 39 (2017),
  pp.~A1225--A1250.

\bibitem{welper19}
{\sc G.~Welper}, {\em Transformed snapshot interpolation with high resolution
  transforms}, SIAM Journal on Scientific Computing, 42 (2020),
  pp.~A2037--A2061.

\bibitem{Yarotsky17}
{\sc D.~Yarotsky}, {\em Error bounds for approximations with deep {R}e{L}{U}
  networks}, Neural Networks, 94 (2017), pp.~103 -- 114.

\end{thebibliography}

\end{document}